\documentclass[a4paper,10pt]{article}
\usepackage{amssymb}
\usepackage{amsmath}

\usepackage{tikz}
\usetikzlibrary{matrix,arrows}

\let\cal=\mathcal

\newcommand{\qed}{
  \ifmmode
   \eqno{\qedsymbol}
  \else
    \leavevmode\unskip\penalty9999 \hbox{}\nobreak\hfill\hbox{\qedsymbol}
  \fi
}
\newcommand{\qedsymbol}{\leavevmode\vrule height 1.2ex width 1.1ex depth -.1ex}
\newenvironment{proof}{\begin{trivlist}\item[\hskip
\labelsep{\bf Proof.\quad}]}
{\hfill\qed\rm\end{trivlist}}

\newtheorem{theorem}{Theorem}[section]
\newtheorem{corollary}[theorem]{Corollary}
\newtheorem{proposition}[theorem]{Proposition}
\newtheorem{lemma}[theorem]{Lemma}

\newtheorem{example}[theorem]{Example}
\newtheorem{remark}[theorem]{Remark}

\mathchardef\emptyset="001F
\font\Bbb=msbm10 at 12 truept
\def\n{\hbox{\Bbb N}}
\def\im{\hbox{\rm im}}
\def\X{{\mathcal X}}
\def\Pr{{\mathcal P}}
\def\SF{{\mathcal {SF}}}
\def\CP{{\mathcal {CP}}}
\def\F{{\mathcal F}}
\def\E{{\mathcal E}}
\def\tensor{\otimes}

\begin{document}

\title{Covers of acts over monoids and pure epimorphisms}
\author{Alex Bailey and James Renshaw\\\small Department of Mathematics\\
\small University of Southampton\\
\small Southampton, SO17 1BJ\\
\small England\\
\small Email: alex.bailey@soton.ac.uk\\
\small j.h.renshaw@maths.soton.ac.uk}
\date{June 2012}
\maketitle

\begin{abstract}
In 2001 Enochs' celebrated flat cover conjecture was finally proven and the proofs (two different proofs were presented in the same paper~\cite{bican-01}) have since generated a great deal of interest among researchers. In particular the results have been recast in a number of other categories and in particular for additive categories (see for example \cite{aldrich-01}, \cite{bashir-04}, \cite{rosicky-02} and \cite{rump-09}). In 2008, Mahmoudi and Renshaw considered a similar problem for acts over monoids but used a slightly different definition of cover. They proved that in general their definition was not equivalent to Enochs', except in the projective case, and left open a number of questions regarding the `other' definition. This `other' definition is the subject of the present paper and we attempt to emulate some of Enochs' work for the category of acts over monoids and concentrate, in the main, on strongly flat acts. We hope to extend this work to other classes of acts, such as injective, torsion free, divisible and free, in a future report.
\end{abstract}



\section{Introduction and Preliminaries}
Let $S$ be a monoid. Throughout, unless otherwise stated, all acts will be right $S-$acts and all congruences right $S-$congruences. We refer the reader to~\cite{howie-95} for basic results and terminology in semigroups and monoids and to  \cite{ahsan-08} and \cite{kilp-00} for those concerning acts over monoids.

Enochs' conjecture, that all modules over a unitary ring, have a flat cover was finally proven in 2001. In 2008, Mahmoudi and Renshaw \cite{renshaw-08}, initiated a study of flat covers of acts over monoids. Their definition of cover concerned coessential epimorphisms and, except for the case of projective covers, proves to be different to that given by Enochs. In the present paper we attempt to initiate the study of Enochs' notion of cover to the category of acts over monoids and focus primarily on $\SF-$covers where $\SF$ is the class of strongly flat $S-$acts.

After preliminary results and definitions we provide some key results on directed colimits for acts over monoids. Some of these may be generally known but there are few references in the literature for results on direct limits of $S-$acts that we felt it necessary to include the more important ones here. Pure epimorphisms were studied by Stenstr\"om in~\cite{stenstrom-71} and we extend these in Section 3. In Section 4 we introduce the concept of an $\X-$cover and $\X-$precover for a class of $S-$acts $\X$. This is the anologue of Enochs' definition for covers of modules over rings and we prove that for those classes that are closed under isomorphisms and directed colimts, the existence of a precover implies the existence of a cover. We also provide a necessary and sufficient condition and a number of sufficient conditions for the existence of a precover. Finally in Section 5 we apply some of these results to the case when $\X$ is the class of strongly flat $S-$acts.

\bigskip

An $S-$act $P$ is called {\em projective} if given any $S-$epimorphism $f:A \rightarrow B$, whenever there is an $S-$map $g:P \rightarrow B$ there exists an $S-$map $h:P \rightarrow A$ such that the following diagram commutes
$$
\begin{tikzpicture}[description/.style={fill=white,inner sep=2pt}]
\matrix (m) [matrix of math nodes, row sep=3em,
column sep=2.5em, text height=1.5ex, text depth=0.25ex]
{A& B\\
 &P \\};
\path[->,font=\scriptsize]
(m-2-2) edge node[auto, right] {$g$} (m-1-2)
(m-2-2) edge node[auto, left] {$h$} (m-1-1)
(m-1-1) edge node[auto, above] {$f$} (m-1-2);
\end{tikzpicture}
$$
A right $S-$act $A$ is said to be {\em flat} if given any monomorphism of left $S-$acts $f:X \rightarrow Y$, the induced map $1 \otimes f:A \otimes_S X \rightarrow A \otimes_S Y$, $a \otimes x \mapsto a \otimes f(a)$, is also a monomorphism. On the other hand, an $S-$monomorphism $g:A\to B$ is said to be {\em pure}, \cite{renshaw-86}, if for all left $S-$acts $X$,  the induced map $A\otimes_S X\to B\otimes_SX$ is also a monomorphism. Note that there are in fact two distinct notions of pure monomorphism in the literature. See~\cite[Section 7.4]{ahsan-08} for more details. In 1969 Lazard proved that flat modules are directed colimits of finitely generated free modules \cite{lazard-69}. In 1971 Stenstr$\ddot{\text{o}}$m showed that the acts which satisfy the same property are different from flat acts \cite{stenstrom-71}. In fact they are the acts $A$ where $A \otimes_S -$ preserves pullbacks and equalizers, or equivalently those that satisfy the two interpolation conditions $(P)$ and $(E)$. These acts have come to be known as {\em strongly flat} acts.

A right $S-$act $A$ is said
to satisfy {\em condition $(P)$} if whenever $au = a'u'$ with $u,u' \in
S, a,a' \in A$, there exists $a''\in A, s,s' \in S$ with $a = a''s, a' =
a''s'$ and $su = s'u'$ whilst $A$ is said to satisfy {\em condition
$(E)$} if whenever $au = au'$ with $a \in A, u,u' \in S$, there exists
$a'' \in A, s \in S$ with $a = a''s$ and $su = su'$. 

Throughout this paper we shall denote the class of all projective $S-$acts by $\Pr_S$, the class of all strongly flat $S-$acts by $\SF_S$, the class of all $S-$acts that satisfy condition $( P)$ by $\CP_S$, the class of all $S-$acts that satisfy condition $(E)$ by $\E_S$ and the class of all flat acts by $\F_S$. Normally we shall simply omit the subscript.

It is well-known that in general
$$
\Pr\subsetneq\SF\subsetneq\CP\subsetneq\F
$$

Basic results on indecomposable acts, coproducts, pushouts and pullbacks of acts over monoids can be found in \cite{kilp-00} and \cite{ahsan-08}. From~\cite[Proposition 4.1.5, Corollary 5.3.23, Proposition 5.2.17 and Proposition 5.2.5]{ahsan-08} (see also~\cite[Lemma III.9.3 \& Lemma III.9.5]{kilp-00}) we have

\begin{lemma} \label{coproduct-decompose}
Let $S$ be a monoid and let $X=\dot\bigcup X_i$ be a coproduct of $S-$acts. For each of the cases $\X=\Pr, \X=\SF,\X=\CP$ and $\X=\F$ we have $X\in\X$ if and only if each $X_i\in\X$.
\end{lemma}

The following lemma will be useful in one of our main results later.
\begin{lemma}\label{indecomposable-epi-lemma}
Let $S$ be a monoid, let $X$ be an indecomposable $S-$act and let $f:X\to Y$ be an $S-$epimorphism. Then $Y$ is indecomposable.
\end{lemma}
\begin{proof}
Suppose that $Y$ is not indecomposable so that there are non-empty $S-$subacts $ Y_1\ne Y_2\subseteq Y$ with $Y=Y_1\dot\cup Y_2$. Then Let $X_i = f^{-1}(Y_i), i = 1,2$ and note that $X_i$ are non-empty $S-$subacts of $X$ and that $X=X_1\dot\cup X_2$ with $X_1\ne X_2$ giving a contradiction.
\end{proof}

Let $A$ be an $S-$act. We say that a projective $S-$act $C$ together with an $S-$epimorphism $f : C \to A$ is a {\em projective cover} of $A$ if there is no proper subact $B$ of $C$ such that $f|_B$ is onto. If we replace `projective' by `strongly flat' in this definition then we have a {\em strongly flat cover}. A monoid $S$ is called {\em perfect} if all $S-$acts have a projective cover.

\smallskip

We define (see \cite{stenstrom-71}~and~\cite{normak-77}) $A$ to be {\em finitely presented} if $A\cong F/\rho$ where $F$ is finitely generated free and $\rho$ is finitely generated.

The following remark will be useful when we come to consider {\em precovers} in section 3.

\begin{remark}\label{remark-1}
Let $S$ be a monoid, let $A$ be an $S-$act and let $\theta$ be a congruence on $A$. Let $\rho$ be a congruence on $A/\theta$ and let $\theta/\rho = \ker(\rho^\natural\theta^\natural)$. Then clearly $\theta/\rho$ is a congruence on $A$ containing $\theta$ and $A/(\theta/\rho) = (A/\theta)/\rho$. Moreover $\theta/\rho=\theta$ if and only if $\rho=1_{F/\theta}$.
\end{remark}

Let $\lambda$ be an infinite cardinal and let $\X$ be a class of $S-$acts. By a {\em $\lambda-$skeleton of $S-$acts}, $\X_\lambda$ we mean a set of pairwise non-isomorphic $S-$acts such that for each act $A\in\X$ with $|A|<\lambda$, there exists a (necessarily unique) act $A_\lambda \in\X_\lambda$ such that $A\cong A_\lambda$.

\begin{remark}\label{bounded-remark}
Let $S$ be a monoid, let $\X$ be a class of $S-$acts and suppose that there exists a cardinal $\lambda$ such that every indecomposable $S-$act $X\in\X$ is such that $|X|<\lambda$. Then it is reasonably clear that the class of indecomposable $S-$acts forms a set and so must contain a $\lambda-$skeleton.
\end{remark}

\section{Colimits and Directed Colimits}

There is surprising little in the literature on direct limits and colimits of acts and in addition some inconsistencies with notation (see \cite{kilp-00} and \cite{renshaw-86}). We include here a collection of results on direct limits, some of which will be needed in later sections.

Let $I$ be a set with a preorder (that is, a reflexive and transitive relation). A {\em direct system} is a collection of $S-$acts $(X_i)_{i \in I}$ together with $S-$maps $\phi_{i,j}:X_i \rightarrow X_j$ for all $i \leq j \in I$ such that
\begin{enumerate}
\item $\phi_{i,i}=1_{X_i}$, for all $i \in I$; and
\item $\phi_{j,k} \circ \phi_{i,j}=\phi_{i,k}$ whenever $i \leq j \leq k$.
\end{enumerate}
The {\em colimit} of the system $(X_i,\phi_{i,j})$ is an $S-$act $X$ together with $S-$maps $\alpha_i:X_i \rightarrow X$ such that
\begin{enumerate}
\item $\alpha_j \circ \phi_{i,j}=\alpha_i$, whenever $i \leq j$,
\item If $Y$ is an $S-$act and $\beta_i:X_i \rightarrow Y$ are $S-$maps such that $\beta_j \circ \phi_{i,j}=\beta_i$ whenever $i \leq j$, then there exists a unique $S-$map $\psi:X \rightarrow Y$ such that the diagram
$$
\begin{tikzpicture}[description/.style={fill=white,inner sep=2pt}]
\matrix (m) [matrix of math nodes, row sep=3em,
column sep=2.5em, text height=1.5ex, text depth=0.25ex]
{X_i& X\\
 Y & \\};
\path[->,font=\scriptsize]
(m-1-1) edge node[auto, left] {$\beta_i$} (m-2-1)
(m-1-2) edge node[auto, right] {$\psi$} (m-2-1)
(m-1-1) edge node[auto, above] {$\alpha_i$} (m-1-2);
\end{tikzpicture}
$$
commutes for all $i \in I$.
\end{enumerate}

If the indexing set $I$ satisfies the property that for all $i,j \in I$ there exists $k \in I$ such that $k\ge i, j$ then we say that $I$ is {\em directed}. In this case we call the colimit a {\em directed colimit}.

As with all universal constructions, the colimit, if it exists, is unique up to isomorphism. That colimits of $S-$acts do indeed exist is easy to demonstrate. In fact let $\lambda_i:X_i\to\dot\bigcup_i{X_i}$ be the natural inclusion and let $\rho$ be the right congruence on $\dot\bigcup_i{X_i}$ generated  by
$$
R=\{(\lambda_i(x_i),\lambda_j(\phi_{i,j}(x_i))) | x_i\in X_i, i\le j\in I\}.
$$
Then $X=\left(\dot\bigcup_i{X_i}\right)/\rho$ and $\alpha_i:X_i\to X$ given by $\alpha_i(x_i)=\lambda_i(x_i)\rho$ are such that $(X,\alpha_i)$ is a colimit of $(X_i,\phi_{i,j})$.  In addition, if the index set $I$ is directed then
$$
\rho=\{(\lambda_i(x_i),\lambda_j(x_j)) | \text{ there exists }k\ge i,j\text{ with }\phi_{i,k}(x_i)=\phi_{j,k}(x_j)\}.
$$
See (\cite[Theorem I.3.1 \& Theorem I.3.17]{renshaw-85}) for more details. We shall subsequently talk of {\em the} (directed) colimit of a direct system.

\begin{lemma}[{\cite[Lemma 3.5 \& Corollary 3.6]{renshaw-86}}]\label{direct-limit-lemma}
Let $(X_i,\phi_{i,j})$ be a direct system of $S-$acts with directed index set and let $(X,\alpha_i)$ be the directed colimit. Then $\alpha_i(x_i)=\alpha_j(x_j)$ if and only if $\phi_{i,k}(x_i) = \phi_{j,k}(x_j)$ for some $k\ge i,j$. Consequently $\alpha_i$ is a monomorphism if and only if $\phi_{i,k}$ is a monomorphism for all $k\ge i$.
\end{lemma}

In fact the following is now easy to establish.

\begin{theorem}
Let $S$ be a monoid, let $(X_i,\phi_{i,j})$ be a direct system of $S-$acts with directed index set $I$ and let $X$ be an $S-$act and $\alpha_i:X_i\to X$ be such that
$$
\begin{tikzpicture}[description/.style={fill=white,inner sep=2pt}]
\matrix (m) [matrix of math nodes, row sep=3em,
column sep=2.5em, text height=1.5ex, text depth=0.25ex]
{X_i& &X_j \\
 & X &\\ };
\path[->,font=\scriptsize]
(m-1-1) edge node[auto, left] {$\alpha_i$} (m-2-2)
(m-1-3) edge node[auto, right] {$\alpha_j$} (m-2-2)
(m-1-1) edge[->] node[auto,above] {$\phi_{i,j}$} (m-1-3);
\end{tikzpicture}
$$
commutes for all $i\le j$ in $I$. Then $(X,\alpha_i)$ is the directed colimit of $(X_i,\phi_{i,j})$ if and only if
\begin{enumerate}
\item for all $x\in X$ there exists $i\in I$ and $x_i\in X_i$ such that $x=\alpha_i(x_i)$,
\item for all $i,j\in I$, $\alpha_i(x_i)=\alpha_j(x_j)$ if and only if $\phi_{i,k}(x_i) = \phi_{j,k}(x_j)$ for some $k\ge i,j$.
\end{enumerate}
\end{theorem}

We shall use these two basic properties of directed colimits without further reference.

\begin{lemma}\label{directlimit-lemma}
Let $S$ be a monoid, let $(X_i,\phi_{i,j})$ be a direct system of $S-$acts with directed index set $I$ and directed colimit $(X,\alpha_i)$. For every family $y_1,\ldots,y_n\in X$ and relations
$$
y_{j_i}s_i=y_{k_i}t_i \quad 1\le i\le m
$$
there exists some $l \in I$ and $x_1,\ldots,x_n\in X_l$ such that $\alpha_l(x_r)=y_r$ for $1\le r\le n$, and
$$
x_{j_i}s_i=x_{k_i}t_i \hbox{ for all }1\le i\le m.
$$
\end{lemma}

\begin{proof}
Given $y_1,\ldots,y_n \in X$ there exists $m(1),\ldots,m(n) \in I$ and $y'_r \in X_{m(r)}$ such that $\alpha_{m(r)}(y'_r)=y_r$ for all $1 \le r \le n$. So for all $1\le i\le m$ we have
$$
\alpha_{m(j_i)}(y'_{j_i}s_i)=\alpha_{m(j_i)}(y'_{j_i})s_i=\alpha_{m(k_i)}(y'_{k_i})t_i=\alpha_{m(k_i)}(y'_{k_i}t_i)
$$
and so there exist $l_i \ge m(j_i),m(k_i)$ such that for all $1\le i\le m$
$$
\phi_{m(j_i),l_i}(y'_{j_i})s_i=\phi_{m(j_i),l_i}(y'_{j_i}s_i)=\phi_{m(k_i),l_i}(y'_{k_i}t_i)=\phi_{m(k_i),l_i}(y'_{k_i})t_i.
$$
Let $l \ge l_1,\ldots,l_m$. Then there exist $\phi_{m(1),l}(y'_1),\ldots,\phi_{m(n),l}(y'_n) \in X_l$ such that $\alpha_l(\phi_{m(r),l}(y'_r))=\alpha_{m(r)}(y'_r)=y_r$ for all $1\le r\le n$ and
$$
\phi_{m(j_i),l}(y'_{j_i})s_i=\phi_{l_i,l}\left(\phi_{m(j_i),l_i}(y'_{j_i})\right)s_i = \phi_{l_i,l}\left(\phi_{m(k_i),l_i}(y'_{k_i})\right)t_i = \phi_{m(k_i),l}(y'_{k_i})t_i
$$
 for all $1\le i\le m$ and the result follows.
\end{proof}

The following result shows that, in a certain sense, directed colimits preserve monomorphisms.

\begin{lemma}\label{directlimit-monomorphism-lemma}
Let $S$ be a monoid, let $(X_i,\phi_{i,j})$ be a direct system of $S-$acts with directed index set and let $(X,\alpha_i)$ be the directed colimit. Suppose that $Y$ is an $S-$act and that $\beta_i:X_i\to Y$ are monomorphisms such that $\beta_i = \beta_j\phi_{i,j}$ for all $i\le j$. Then there exists a unique monomorphism $h:X\to Y$ such that $h\alpha_i=\beta_i$ for all $i$.
\end{lemma}

\begin{proof}
Consider the following commutative diagram
$$
\begin{tikzpicture}[description/.style={fill=white,inner sep=2pt}]
\matrix (m) [matrix of math nodes, row sep=3em,
column sep=2.5em, text height=1.5ex, text depth=0.25ex]
{X_i& &X_j \\
 & X &\\ 
& Y & \\};
\path[->,font=\scriptsize]
(m-1-1) edge node[auto, right] {$\alpha_i$} (m-2-2)
(m-1-3) edge node[auto, left] {$\alpha_j$} (m-2-2)
(m-1-1) edge node[auto, left] {$\beta_i$} (m-3-2)
(m-1-3) edge node[auto, right] {$\beta_j$} (m-3-2)
(m-2-2) edge node[auto, above left] {$h$} (m-3-2)
(m-1-1) edge[->] node[auto,above] {$\phi_{i,j}$} (m-1-3);
\end{tikzpicture}
$$
where $h$ is the unique map guaranteed by the directed colimit property. Suppose that $h(x)=h(x')$. Then there exists $i,j$ and $x_i\in X_i, x_j\in X_j$ such that $x=\alpha_i(x_i)$ and $x'=\alpha_j(x_j)$. Hence there exists $k\ge i,j$ and so
$$
\beta_k\phi_{i,k}(x_i) = h\alpha_k\phi_{i,k}(x_i)=h\alpha_i(x_i)=h\alpha_j(x_j) = h\alpha_k\phi_{j,k}(x_j)=\beta_k\phi_{j,k}(x_j).
$$
Since $\beta_k$ is a monomorphism then $\phi_{i,k}(x_i)=\phi_{j,k}(x_j)$ and so $x=x'$ as required.
\end{proof}

\begin{lemma}\label{direct-limit-quotient-lemma}
Let $S$ be a monoid, let $X$ be an $S-$act and let $\{\rho_i : i \in I\}$ be a set of congruences on $X$, partially ordered by inclusion, with the property that the index set is directed and has a minimum element 0. Let $\phi_{i,j}:X/\rho_i \rightarrow X/\rho_j$ be the $S-$map defined by $a\rho_i \mapsto a\rho_j$ whenever $\rho_i \subseteq \rho_j$, so that $(X/\rho_i,\phi_{i,j})$ is a direct system. Let $\rho=\bigcup_{i \in I}\rho_i$. Then $X/\rho$ is the directed colimit of $(X/\rho_i,\phi_{i,j})$.
\end{lemma}

\begin{proof}
First note that $\rho$ is transitive since $I$ is directed. Clearly we can define $S-$maps $\alpha_i:X/\rho_i \rightarrow X/\rho$, $a\rho_i \mapsto a\rho$ such that $\alpha_i=\alpha_j \phi_{i,j}$ for all $i \leq j$. Now suppose there exists an $S-$act $Q$ and $S-$maps $\beta_i:X/\rho_i \rightarrow Q$ such that $\beta_i=\beta_j\phi_{i,j}$ for all $i \leq j$. Define $\psi:X/\rho \rightarrow Q$ by $\psi(a\rho)=\beta_0(a\rho_0)$. To see this is well-defined, let $a\rho=a'\rho$ in $X/\rho$, that is, $(a,a') \in \rho$ so there must exist some $k \in I$ such that $(a,a') \in \rho_k$ and we get
$$
\beta_0(a\rho_0)=\beta_k\phi_{0,k}(a\rho_0)=\beta_k(a\rho_k)=\beta_k(a'\rho_k)=\beta_k\phi_{0,k}(a'\rho_0)=\beta_0(a'\rho_0)
$$
so $\psi(a\rho)=\psi(a'\rho)$ and $\psi$ is well-defined. It is easy to see that $\psi$ is also an $S-$map. Because $0$ is the minimum element, we have that $\beta_0(a\rho_0)=\beta_0\phi_{i,0}(a\rho_i)=\beta_i(a\rho_i)$ and so $\psi\alpha_i=\beta_i$ for all $i \in I$. Finally let $\psi':X/\rho \rightarrow Q$ be an $S-$map such that $\psi'\alpha_i=\beta_i$ for all $i \in I$, then $\psi'(a\rho)=\psi'(\alpha_0(a\rho_0))=\beta_0(a\rho_0)=\psi(a\rho)$, and we are done.
\end{proof}

\begin{remark}
In particular, this holds when we have a chain of congruences $\rho_1\subset\rho_2\subset\ldots$ and $\rho=\bigcup_{i \geq 1}\rho_i$.
\end{remark}

\begin{example} \label{min-gp-cong}
\rm If $S$ is an inverse monoid, which we consider as a right $S-$act, then for any $e\le f \in E(S)$ it follows that $\ker \lambda_f \subseteq \ker \lambda_e$, where $\lambda_e(s)=es$. Hence there is a set of right congruences on $S$ partially ordered by inclusion, where the identity relation $\ker \lambda_1$ is a least element in the ordering. We can now construct a direct system of $S-$acts $S/\ker \lambda_f \rightarrow S/\ker \lambda_e$, $s\ker \lambda_f \mapsto s\ker \lambda_e$ whose directed colimit, by the previous lemma, is $S/\sigma$ where $\sigma=\bigcup_{e \in E(S)}\ker \lambda_e$, which is easily seen to be the minimum group congruence on $S$ (see~\cite[page 159]{howie-95}).
\end{example}

\begin{proposition}[{\cite[Proposition 5.2]{stenstrom-71}}]\label{direct-limit-sf-proposition}
Let $S$ be a monoid. Every directed colimit of a direct system of strongly flat acts is strongly flat.
\end{proposition}

The following is probably well-known.

\begin{proposition}\label{direct-limit-p-proposition}
Let $S$ be a monoid. Every directed colimit of a direct system of acts that satisfy condition $(P)$, satisfies condition $(P)$.
\end{proposition}
\begin{proof}
Let $(X_i,\phi_{i,j})$ be a direct system of $S-$acts and $S-$morphisms with a directed index set and with directed colimit $(X,\alpha_i)$. Suppose that $xs=yt$ in $X$ so that there exists $x_i\in X_i, x_j\in X_j$ with $x=\alpha_i(x_i), y=\alpha_j(x_j)$. Then by Lemma~\ref{direct-limit-lemma} there exists $k\ge i,j$ with $\phi_{i,k}(x_i)s=\phi_{j,k}(x_j)t$ in $X_k$. Consequently there exists $z\in X_k, u,v\in S$ with $\phi_{i,k}(x_i)=zu, \phi_{j,k}(x_j)=zv$ and $us=vt$. But then $x=\alpha_i(x_i) = \alpha_k\phi_{i,k}(x_i) = \alpha_k(z)u$. In a similar way $y=\alpha_k(z)v$ and the result follows.
\end{proof}

The situation for projective acts is slightly different.

\begin{proposition}[{\cite{fountain-76}}]
Let $S$ be a monoid. Every directed colimit of a direct system of projective acts is projective if and only if $S$ is perfect.
\end{proposition}

\section{Purity and epimorphisms}
Let $\psi:X\to Y$ be an $S-$epimorphism. We say that $\psi$ is {\em a pure epimorphism} if for every finitely presented $S-$act $M$ and every $S-$map $f:M\to Y$ there exists $g:M\to X$ such that
$$
\begin{tikzpicture}[description/.style={fill=white,inner sep=2pt}]
\matrix (m) [matrix of math nodes, row sep=3em,
column sep=2.5em, text height=1.5ex, text depth=0.25ex]
{X & Y \\
 & M\\ };
\path[->,font=\scriptsize]
(m-2-2) edge node[auto, left] {$g$} (m-1-1)
(m-2-2) edge node[auto,right] {$f$} (m-1-2)
(m-1-1) edge[->] node[auto,above] {$\psi$} (m-1-2);
\end{tikzpicture}
$$
commutes.
\begin{theorem}[{\cite[Proposition 4.3]{stenstrom-71}}]\label{pure-theorem}
Let $S$ be a monoid and let $\psi:X\to Y$ be an $S-$epimorphism. Then the following are equivalent:

\begin{enumerate}
\item $\psi$ is pure;
\item for every family $y_1,\ldots,y_n\in Y$ and relations
$$
y_{j_i}s_i=y_{k_i}t_i \quad (1\le i\le m)
$$
there exists $x_1,\ldots,x_n\in X$ such that $\psi(x_r)=y_r$ for $1\le r\le n$ and
$$
x_{j_i}s_i=x_{k_i}t_i \hbox{ for all }1\le i\le m.
$$
\end{enumerate}
\end{theorem}

\begin{example}\label{min-gc-pure-example}
\rm Let $S$ be an inverse monoid and $\sigma$ the minimum group congruence on $S$ as in Example~\ref{min-gp-cong}. Then the right $S-$map $S\to S/\sigma$ is a pure $S-$epimorphism. To see this let $y_1=x_1\sigma,\ldots,y_n=x_n\sigma\in S/\sigma$ and suppose we have relations
$$
y_{j_i}s_i=y_{k_i}t_i \quad (1\le i\le m).
$$
Then for $1\le i\le m$ we have $(x_{j_i}s_i,x_{k_i}t_i)\in\sigma$ and so there exist $e_i\in E(S), (1\le i\le m)$ such that
$e_ix_{j_i}s_i=e_ix_{k_i}t_i$. Now let $e=e_1\ldots e_m$ and note that for $1\le i\le m, ex_{j_i}s_i = ex_{k_i}t_i$ and for $1\le l\le n, \sigma^\natural(ex_l) = (ex_l)\sigma = x_l\sigma=y_l$ as required.

\end{example}

\medskip

It is clear that if the epimorphism $\psi$ splits with splitting monomorphism $\phi:Y\to X$ then $\phi f:M\to X$ is such that $\psi\phi f = f$ and so $\psi$ is pure. The converse is not in general true. For example, let $S=\n$ with multiplication given by
$$
n.m=\max\{m,n\}\text{ for all }m,n\in S.
$$
Let $\Theta_S=\{\theta\}$ be the 1-element right $S-$act and note that $S\to \Theta_S$ is a pure epimorphism by Theorem~\ref{pure-theorem}. However, as $S$ does not contain a fixed point then it does not split.

\medskip

From Lemma~\ref{directlimit-lemma} we can immediately deduce that
\begin{corollary}
Let $S$ be a monoid, let $(X_i,\phi_{i,j})$ be a direct system of $S-$acts with directed index set $I$ and directed colimit $(X,\alpha_i)$. Then the natural map $\dot\bigcup X_i\to X$ is a pure epimorphism.
\end{corollary}
Suppose that $(X_i,\phi_{i,j})$ and $(Y_i,\theta_{i,j})$ are direct systems of $S-$acts and $S-$maps and suppose that for each $i\in I$ there exists an $S-$map $\psi:X_i\to Y_i$  and suppose $(X,\beta_i)$ and $(Y,\alpha_i)$, the directed colimits of these systems are such that
$$
\begin{tikzpicture}[description/.style={fill=white,inner sep=2pt}]
\matrix (m) [matrix of math nodes, row sep=3em,
column sep=2.5em, text height=1.5ex, text depth=0.25ex]
{X_i & Y_i \\
 X & Y\\ };
\path[->,font=\scriptsize]
(m-1-1) edge node[auto, left] {$\beta_i$} (m-2-1)
(m-1-2) edge node[auto,right] {$\alpha_i$} (m-2-2)
(m-2-1) edge[->] node[auto,below] {$\psi$} (m-2-2)
(m-1-1) edge[->] node[auto,above] {$\psi_i$} (m-1-2);
\end{tikzpicture}
\hskip2em
\begin{tikzpicture}[description/.style={fill=white,inner sep=2pt}]
\matrix (m) [matrix of math nodes, row sep=3em,
column sep=2.5em, text height=1.5ex, text depth=0.25ex]
{X_i & X_j \\
 Y_i & Y_j\\ };
\path[->,font=\scriptsize]
(m-1-1) edge node[auto, left] {$\psi_i$} (m-2-1)
(m-1-2) edge node[auto,right] {$\psi_j$} (m-2-2)
(m-2-1) edge[->] node[auto,below] {$\theta_{i,j}$} (m-2-2)
(m-1-1) edge[->] node[auto,above] {$\phi_{i,j}$} (m-1-2);
\end{tikzpicture}
$$
commute for all $i\le j\in I$. Then we shall refer to $\psi$ as the {\em directed colimit of the $\psi_i$}.
It is shown in \cite{renshaw-85} that directed colimits of (monomorphisms) epimorphisms are (monomorphisms) epimorphisms.

\begin{corollary}\label{pure-epimorphism-corollary}
Let $S$ be a monoid. Directed colimits of pure $S-$epimorphisms are pure.
\end{corollary}
\begin{proof}
Suppose that $(X_i,\phi_{i,j})$ and $(Y_i,\theta_{i,j})$ are direct systems and for each $i\in I$ there exists a pure epimorphism $\psi:X_i\to Y_i$  and suppose $(X,\beta_i)$ and $(Y,\alpha_i)$, the directed colimits of these systems are such that
$$
\begin{tikzpicture}[description/.style={fill=white,inner sep=2pt}]
\matrix (m) [matrix of math nodes, row sep=3em,
column sep=2.5em, text height=1.5ex, text depth=0.25ex]
{X_i & Y_i \\
 X & Y\\ };
\path[->,font=\scriptsize]
(m-1-1) edge node[auto, left] {$\beta_i$} (m-2-1)
(m-1-2) edge node[auto,right] {$\alpha_i$} (m-2-2)
(m-2-1) edge[->] node[auto,below] {$\psi$} (m-2-2)
(m-1-1) edge[->] node[auto,above] {$\psi_i$} (m-1-2);
\end{tikzpicture}
\hskip2em
\begin{tikzpicture}[description/.style={fill=white,inner sep=2pt}]
\matrix (m) [matrix of math nodes, row sep=3em,
column sep=2.5em, text height=1.5ex, text depth=0.25ex]
{X_i & X_j \\
 Y_i & Y_j\\ };
\path[->,font=\scriptsize]
(m-1-1) edge node[auto, left] {$\psi_i$} (m-2-1)
(m-1-2) edge node[auto,right] {$\psi_j$} (m-2-2)
(m-2-1) edge[->] node[auto,below] {$\theta_{i,j}$} (m-2-2)
(m-1-1) edge[->] node[auto,above] {$\phi_{i,j}$} (m-1-2);
\end{tikzpicture}$$
commute for all $i\le j\in I$.

Suppose there are $y_1,\ldots,y_n\in Y, s_1,\ldots, s_m,t_1,\ldots t_m\in S$ and relations
$$
y_{j_i}s_i=y_{k_i}t_i \quad (1\le i\le m).
$$
By Lemma~\ref{directlimit-lemma} there exists $l \in I$ and $z_1,\ldots,z_n\in Y_l$ such that $\alpha_l(z_r)=y_r$ for $1\le r\le n$, and
$$
z_{j_i}s_i=z_{k_i}t_i \hbox{ for all }1\le i\le m.
$$
Since $\psi_l$ is pure there exist $x_1,\ldots,x_n\in X_l$ such that $\psi_l(x_r)=z_r$ for $1\le r\le n$, and
$$
x_{j_i}s_i=x_{k_i}t_i \hbox{ for all }1\le i\le m.
$$
Hence
$$
\beta_l(x_{j_i})s_i=\beta_l(x_{k_i})t_i \hbox{ for all }1\le i\le m.
$$
and $\psi\beta_l(x_r)=\alpha_l\psi_l(x_r)=\alpha_l(z_r) = y_r$ for $1\le r\le n$ and so $\psi$ is pure.
\end{proof}

\begin{lemma}\label{pullback-pure-lemma}
Let $S$ be a monoid, let
$$
\begin{tikzpicture}[description/.style={fill=white,inner sep=2pt}]
\matrix (m) [matrix of math nodes, row sep=3em,
column sep=2.5em, text height=1.5ex, text depth=0.25ex]
{A & B \\
 C & D\\ };
\path[->,font=\scriptsize]
(m-1-1) edge node[auto, left] {$\alpha$} (m-2-1)
(m-1-2) edge node[auto,right] {$\beta$} (m-2-2)
(m-2-1) edge[->] node[auto,below] {$\psi$} (m-2-2)
(m-1-1) edge[->] node[auto,above] {$\phi$} (m-1-2);
\end{tikzpicture}
$$
be a pullback diagram of $S-$acts and suppose that $\psi$ is a pure epimorphism. Then $\phi$ is also a pure epimorphism.
\end{lemma}

\begin{proof}
That $\phi$ is onto is clear. Suppose that $M$ is finitely presented and that $f:M\to B$ is a morphism. Then there exists $g:M\to C$ such that $\psi g=\beta f$. Since $A$ is a pullback then there exists a unique $h:M\to A$ such that $\phi h=f$ and $\alpha h = g$.
\end{proof}

Although not every pure epimorphism splits, we can deduce
\begin{theorem}
Let $S$ be a monoid and let $\psi:X\to Y$ be an epimorphism. Then $\psi$ is pure if and only if it is a directed colimit of split epimorphisms.
\end{theorem}
\begin{proof}
Suppose that $\psi$ is pure. We know (\cite[Proposition 4.1]{stenstrom-71}) that $Y$ is a directed colimit of finitely presented acts $(Y_i,\phi_{i,j})$ and so let $\alpha_i:Y_i\to Y$ be the canonical maps. For each $Y_i$ let
$$
\begin{tikzpicture}[description/.style={fill=white,inner sep=2pt}]
\matrix (m) [matrix of math nodes, row sep=3em,
column sep=2.5em, text height=1.5ex, text depth=0.25ex]
{X_i & Y_i \\
 X & Y\\ };
\path[->,font=\scriptsize]
(m-1-1) edge node[auto, left] {$\beta_i$} (m-2-1)
(m-1-2) edge node[auto,right] {$\alpha_i$} (m-2-2)
(m-2-1) edge[->] node[auto,below] {$\psi$} (m-2-2)
(m-1-1) edge[->] node[auto,above] {$\psi_i$} (m-1-2);
\end{tikzpicture}
$$
be a pullback diagram so that by Lemma~\ref{pullback-pure-lemma} $\psi_i$ is pure. Hence since $Y_i$ is finitely presented then it easily follows that $\psi_i$ splits. 
Notice that $X_i=\{(y_i,x)\in Y_i\times X | \alpha_i(y_i)=\psi(x)\}, \psi_i(y_i,x) = y_i$ and $\beta_i(y_i,x) = x$ and that since $\psi$ is onto then $X_i\ne\emptyset$.

For $i\le j$ define $\theta_{i,j}:X_i\to X_j$ by $\theta_{i,j}(y_i,x)=(\phi_{i,j}(y_i),x)$ and notice that $\beta_j\theta_{i,j}=\beta_i$ and that $\psi_j\theta_{i,j} =\phi_{i,j}\psi_i$.
Suppose now that there exists $Z$ and $\gamma_i:X_i\to Z$ with $\gamma_j\theta_{i,j}=\gamma_i$ for all $i\le j$.
Define $\gamma:X\to Z$ by $\gamma(x)=\gamma_i(y_i,x)$ where $i$ and $y_i$ are chosen so that $\alpha_i(y_i) = \psi(x)$. Then $\gamma$ is well-defined since if $\psi(x)=\alpha_j(y_j)$ then there exists $k\ge i,j$ with $\phi_{i,k}(y_i)=\phi_{j,k}(y_j)$ and
$$
\begin{array}{rl}
\gamma_i(y_i,x) &=\gamma_k\theta_{i,k}(y_i,x)\\
&=\gamma_k(\phi_{i,k}(y_i),x)\\
&=\gamma_k(\phi_{j,k}(y_j),x)\\
&=\gamma_k\theta_{j,k}(y_j,x)\\
&=\gamma_j(y_j,x).\\
\end{array}
$$
Then $\gamma$ is an $S-$map and clearly $\gamma\beta_i = \gamma_i$. Finally, if $\gamma':X\to Z$ is such that $\gamma'\beta_i=\gamma_i$ for all $i$, then $\gamma'(x) = \gamma'\beta_i(y_i,x) =\gamma_i(y_i,x) = \gamma(x)$ and so $\gamma$ is unique. We therefore have that $(X,\beta_i)$ is the directed colimit of $(X_i,\theta_{i,j})$ as required.

Conversely, since split epimorphisms are pure then $\psi$ is pure by Corollary~\ref{pure-epimorphism-corollary}.
\end{proof}

\begin{example}
\rm Let $S$ be as in Example \ref{min-gp-cong}. Notice that for all $e\in E(S)$, $S\to S/\ker\lambda_e$ splits with splitting map $s\ker\lambda_e\mapsto es$. Moreover
$$
\begin{tikzpicture}[description/.style={fill=white,inner sep=2pt}]
\matrix (m) [matrix of math nodes, row sep=3em,
column sep=2.5em, text height=1.5ex, text depth=0.25ex]
{S & S/\ker\lambda_e \\
 S & S/\sigma\\ };
\path[->,font=\scriptsize]
(m-1-1) edge node[auto, left] {$1_S$} (m-2-1)
(m-1-2) edge node[auto,right] {} (m-2-2)
(m-2-1) edge[->] node[auto,below] {$\sigma^\natural$} (m-2-2)
(m-1-1) edge[->] node[auto,above] {} (m-1-2);
\end{tikzpicture}
$$
commutes for all $e\in E(S)$ and $\sigma^\natural$ is a directed colimit of split epimorphisms.
\end{example}

\begin{theorem}[{\cite[Theorem 5.3]{stenstrom-71}}]\label{stenstrom-theorem}
Let $S$ be a monoid. Then an $S-$act $Y$ is strongly flat if and only if every epimorphism $X\to Y$ is pure.
\end{theorem}

In~\cite{normak-87}, Normak defines an epimorphism $\phi:X\to Y$ to be {\em $1-$pure} if for every element $y\in Y$ and relations $ys_i = yt_i, i=1,\ldots, n$ there exists an element $x\in X$ such that $\phi(x)=y$ and $xs_i=xt_i$ for all $i$. He proves

\begin{proposition}[{\cite[Proposition 1.17]{normak-87}}]
Let $S$ be a monoid. An epimorphism $\phi:X\to Y$ is $1-$pure if and only if for all cyclic finitely presented $S-$acts $C$ and every morphism $f:C\to Y$ there exits $g:C\to X$ with $f=\phi g$.
\end{proposition}

\begin{proposition}[{\cite[Proposition 2.2]{normak-87}}]\label{1-pure-proposition}
Let $S$ be a monoid. $Y$ satisfies condition $(E)$ if and only if every epimorphism $X\to Y$ is $1-$pure.
\end{proposition}

As a generalisation, we say that an epimorphism $g:B \rightarrow A$ of $S$-acts is {\em $n$-pure} if for every family of $n$ elements $a_1,\ldots,a_n \in A$ and every finite family of relations $a_{\alpha_i}s_i=a_{\beta_i}t_i$, $i=1,\ldots,m$, there exist $b_1,\ldots,b_n \in B$ such that $g(b_i)=a_i$ and $b_{\alpha_i}s_i=b_{\beta_i}t_i$ for all $i$.

We are interested in the cases $n=1$ and $n=2$. Clearly pure implies $2-$pure implies $1-$pure.

\begin{proposition}\label{ce-pure-proposition}
Let $S$ be a monoid and let $\psi : X\to Y$ be an $S-$epimorphism in which $X$ satisfies condition $(E)$. Then $Y$ satisfies condition $(E)$ if and only if $\psi$ is $1-$pure.
\end{proposition}
\begin{proof}
Suppose that $\psi$ is $1-$pure and that $y\in Y,s,t\in S$ are such that $ys=yt$ in $Y$.
 Hence there exists $x\in X$ such that $\psi(x)=y$ and $xs=xt$. Since $X$ satisfies condition $(E)$ there exists $x'\in X, u\in S$ such that $x=x'u, us=ut$ and so $y=\psi(x')u, us=ut$ and $Y$ satisfies condition $(E)$.

The converse holds by Proposition \ref{1-pure-proposition}.
\end{proof}

\begin{proposition}\label{cp-pure-proposition}
Let $S$ be a monoid and let $\psi : X\to Y$ be an $S-$epimorphism in which $X$ satisfies condition $(P)$. If $\psi$ is $2-$pure  then $Y$ satisfies condition $(P)$.
\end{proposition}
\begin{proof}
Suppose that $\psi$ is $2-$pure and suppose that $y_1,y_2\in Y,s_1,s_2\in S$ are such that $y_1s_1=y_2s_2$ in $Y$. Hence there exists $x_1,x_2\in X$ with $\psi(x_i)=y_i$ and $x_1s_1=x_2s_2$ in $X$. Since $X$ satisfies condition $(P)$ then there exists $x_3\in X, u_1,u_2\in S$ such that $x_1=x_3u_1, x_2=x_3u_2$ and $u_1s_1=u_2s_2$. Consequently, $y_1=\psi(x_3)u_1, y_2=\psi(x_3)u_2$ and $u_1s_1=u_2s_2$ and so $Y$ satisfies condition $(P)$.
\end{proof}

The converse of this last result is false. For example let $S=(\n\cup\{0\},+)$ and let $\Theta_S=\{\theta\}$ be the 1-element $S-$act. Let $x=y=\theta \in \Theta_S$, then $x0=y0$ and $x0=y1$ but there cannot exist $x',y' \in S$ such that $x'+0=y'+0$ and $x'+0=y'+1$ and so $S\to\Theta_S$ is not $2-$pure, but it is easy to check that $\Theta_S$ does satisfy condition $(P)$.

\medskip

From Theorem~\ref{stenstrom-theorem}, Proposition~\ref{ce-pure-proposition} and Proposition~\ref{cp-pure-proposition} we deduce
\begin{corollary}\label{x-pure-corollary}
Let $S$ be a monoid and let $\psi:X\to Y$ be an $S-$epimorphism with $X$ strongly flat. The following are equivalent.
\begin{enumerate}
\item $Y$ is strongly flat;
\item $\psi$ is pure;
\item $\psi$ is $2-$pure.
\end{enumerate}
\end{corollary}

Let $X$ be an $S-$act and $\theta$ a congruence on $X$. Say that $\theta$ is {\em pure} if $X\to X/\theta$ is pure. As a corollary to Theorem~\ref{pure-theorem} we have
\begin{corollary}\label{pure-congruence-corollary}
Let $S$ be a monoid, let $X$ be an $S-$act and $\theta$ a congruence on $X$. Then $\theta$ is pure if and only if for every family $x_1\ldots,x_n\in X$ and relations
$$
x_{j_i}s_i\;\theta\; x_{k_i}t_i \quad (1\le i\le m)
$$
on $X$ there exists $y_1,\ldots,y_n\in X$ such that $y_i\theta x_i$ and
$$
y_{j_i}s_i=y_{k_i}t_i \hbox{ for all }1\le i\le m.
$$
\end{corollary}

\begin{corollary}
Let $\rho$ be a right $S-$congruence on a monoid $S$. Then $\rho$ is pure if and only if $S/\rho$ is strongly flat.
\end{corollary}

\begin{example}
\rm It now follows easily from Example~\ref{min-gc-pure-example} that if $S$ is an inverse monoid with minimum group congruence $\sigma$ then $S/\sigma$ is a strongly flat right $S-$act.
\end{example}

\medskip

Let $f:X\to Y$ be an $S-$monomorphism. Then Renshaw~\cite{renshaw-02} defined $f$ to be {\em $P-$unitary} if
$$
(\forall y,y'\in Y\backslash\im(f))(\forall s,t\in S)\quad ys,y't\in\im(f)\Rightarrow ys=y't.
$$
This is obviously equivalent to whenever $y,y'\in Y, s,t\in S$ are such that $ys\ne y't$ but $ys,y't\in\im(f)$ then either $y\in\im(f)$ or $y'\in\im(f)$.
\smallskip

In the same way he defined $f$ to be {\em $E-$unitary} if
$$
(\forall y\in Y\backslash\im(f))(\forall s,t\in S)\quad ys,yt\in\im(f)\Rightarrow ys=yt,
$$
which is obviously equivalent to whenever $y\in Y, s,t\in S$ are such that $ys\ne yt$ but $ys,yt\in\im(f)$ then $y\in\im(f)$.

\begin{theorem}
Let $S$ be a monoid, let $f:X\to Y$ be a monomorphism and suppose that $Y\to Y/X$ is a $2-$pure epimorphism. Then $f$ is $P-$unitary. Moreover for all $s,t\in S$ there exists $x,x'\in X$ with $xs=x't$.
\end{theorem}
\begin{proof}
Let $\rho=\im(f)\times\im(f)\cup 1_Y$ so that $Y/X=Y/\rho$. Let $y,y'\in Y\backslash\im(f)$ and suppose that $ys,y't\in\im(f)$. Then $ys\;\rho\;y't$ and so by assumption it easily follows that $ys=y't$ as required.

Let $x\in X$ so that $f(x)s\;\rho\;f(x)t$. Then there exists $x_1,x_2\in X$ with
$$
f(x_1)\;\rho\;f(x)\;\rho\;f(x_2)\text{ and }f(x_1)s=f(x_2)t
$$ Hence $x_1s=x_2t$ as required.
\end{proof}

It then follows from~\cite[Theorem 4.1 and Theorem 4.3]{renshaw-02} that if $Y\to Y/X$ is a pure epimorphism then $f:X\to Y$ is a pure monomorphism. In fact following the remark after the proof of~\cite[Theorem 4.1]{renshaw-02} we see that $f$ splits. In addition we see from ~\cite[Theorem 4.22]{renshaw-02} that if every epimorphism is pure then $S$ is a group. Actually, from Theorem~\ref{stenstrom-theorem} we see that all $S-$acts are strongly flat and so $S$ is the trivial group.

\begin{theorem}
Let $S$ be a monoid, let $f:X\to Y$ be a monomorphism and suppose that $Y\to Y/X$ is a $1-$pure epimorphism. Then $f$ is $E-$unitary. Moreover for all $s,t\in S$ there exists $x\in X$ with $xs=xt$.
\end{theorem}
\begin{proof}
Let $\rho=\im(f)\times\im(f)\cup 1_Y$ so that $Y/X=Y/\rho$. Let $y\in Y\setminus\im(f), s,t\in S$ and suppose that $ys,yt\in\im(f)$. Then $(y\rho)s=(y\rho)t$ and so there exists $z\in Y, u\in S$ with $y\rho=(z\rho)u$ and $us=ut$. Hence $y=zu$ and so $ys=yt$ as required.

Let $x\in X$ so that $f(x)\rho s=f(x)\rho t$. Then there exists $y\in Y$ with $y\rho=f(x)\rho$ and $ys=yt$. Hence $y=x_1$ for some $x_i\in X$ and so $x_1s=x_1t$ as required.
\end{proof}

\begin{theorem}
Let $S$ be a monoid, let $f:X\to Y$ be a monomorphism and suppose that $Y\to Y/X$ is a split epimorphism. Then $f$ is $P-$unitary. Moreover for all $s,t\in S$ there exists $x\in X$ with $xs=xt$.
\end{theorem}
\begin{proof}
Let $\rho=\im(f)\times\im(f)\cup 1_Y$ so that $Y/X=Y/\rho$. Let $g:Y/X\to Y$ be the splitting map. Notice that if $y\not\in\im(f)$ then $g(y\rho)=y$. Let $y,y'\in Y\setminus\im(f), s,t\in S$ and suppose that $ys,y't\in\im(f)$. Then $(y\rho)s=(y'\rho)t$ and so $g(y\rho)s=g(y'\rho)t$. Consequently $ys=y't$ as required.

Let $x\in X$ so that $f(x)\rho s=f(x)\rho t$. Then $g(f(x)\rho)s=g(f(x)\rho)t$ and so there exists $x_1\in X$ with $g(f(x)\rho) = f(x_1)$ and so $x_1s=x_1t$ as required.
\end{proof}
\medskip

\section{Covers and Precovers}

Let $S$ be a monoid, and $A$ be an $S-$act. Unless otherwise stated, in the rest of this section, $\X$ will be a class of $S-$acts closed under isomorphisms. By an $\X$-{\em precover} of $A$ we mean an $S-$map $g: P\to A$ for some $P\in \X$ such that
for every $S-$map $g':P'\to A$, for $P'\in \X$, there exists an $S-$map $f:P'\to P$ with $g'=gf$.
$$
\begin{tikzpicture}[description/.style={fill=white,inner sep=2pt}]
\matrix (m) [matrix of math nodes, row sep=3em,
column sep=2.5em, text height=1.5ex, text depth=0.25ex]
{P & A\\ 
&P'\\};
\path[->,font=\scriptsize]
(m-2-2) edge node[auto,below left] {$f$} (m-1-1)
(m-2-2) edge node[auto, right] {$g'$} (m-1-2)
(m-1-1) edge[->] node[auto,above] {$g$} (m-1-2);
\end{tikzpicture}
$$

If in addition the precover satisfies the condition that each $S-$map $f:P\to P$ with $gf=g$ is an isomorphism, then we shall call it an {\em $\X-$cover}. We shall of course frequently identify the (pre)cover with its domain. Obviously an $S-$act, $A$, is an $\X$-cover of itself if and only if $A \in \X$.
Note that this definition of cover is different from that given in~\cite{renshaw-08}.

\begin{theorem}[{\cite[Theorem 5.8]{renshaw-08}}]
Let $S$ be a monoid. If $g_1 : X_1\to A$ and $g_2:X_2\to A$ are both ${\X}-$covers of an $S-$act $A$ then there is an isomorphism $h : X_1\to X_2$ such that $g_2h=g_1$.
\end{theorem}

\smallskip

\begin{theorem}[{\cite[Theorem 5.7]{renshaw-08}}]
Let $S$ be a monoid. An $S-$map $g:P\to A$, with $P\in \Pr$, is a $\Pr$-cover of $A$ if and only if it is a projective cover.
\end{theorem}

It was demonstrated in~\cite{renshaw-08} that the previous result is not true for condition $(P)$. We show in Section 5 that it is also false for strongly flat acts.

\medskip

Recall from \cite[Theorem II.3.16]{kilp-00} that an $S-$act $G$ is called a {\em generator} if there exists an $S-$epimorphism $G \to S$.

\begin{proposition}
Let $S$ be a monoid and let $\X$ be a class of $S-$acts which contains a generator $G$. If $g:C \rightarrow A$ is an $\X$-precover of $A$ then $g$ is an epimorphism.
\end{proposition}

\begin{proof}
Let $h:G \to S$ be an $S-$epimorphism. Then there exists an $x \in G$ such that $h(x)=1$. For all $a \in A$ define the $S-$map $\lambda_a:S \rightarrow A$ by $\lambda_a(s)=as$. By the $\X$-precover property there exists an $S-$map $f:G \rightarrow C$ such that $gf=\lambda_a h$. Hence $g(f(x))=a$ and so $\im(g)=A$ and $g$ is epimorphic.
\end{proof}

Obviously if every $S-$act has an epimorphic $\X-$precover, then $S$ has an epimorphic $\X-$precover, which by definition is then a generator in $\X$, so we have the following corollary.

\begin{corollary}
Let $S$ be a monoid and $\X$ a class of $S-$acts such that every $S-$act has an $\X-$precover. Then every $S-$act has an epimorphic $\X-$precover if and only if $\X$ contains a generator.
\end{corollary}

Note that for any class of $S-$acts containing $S$ then $S$ is a generator in $\X$ and so $\X$-precovers are always epimorphic. In particular this is true for the classes $\Pr, \SF, \CP$ and $\F$.

\begin{lemma} \label{disjoint-hom}
Let $S$ be a monoid and let $h:X \rightarrow A$ be a homomorphism of $S-$acts where $A=\dot\bigcup_{i \in I}A_i$ is a coproduct of non-empty subacts $A_i \subseteq A$. Then there exists $J\subseteq I$ and $X_j\subseteq X$ for each $j\in J$ such that $X=\dot\bigcup_{j \in J}X_j$ and $\im(h|_{X_j}) \subseteq A_j$ for each $j \in J$. Moreover, if $h$ is an epimorphism, then $J=I$.
\end{lemma}

\begin{proof}
For each $i\in I$ let $X_i=\{x \in X : h(x) \in A_i\}$ and define $J=\{i \in I : X_i \neq \emptyset\}$. For all $x_j \in X_j$, $s \in S$, $h(x_j s)=h(x_j)s \in A_j$ and so $x_j s \in X_j$ and $X_j$ is a subact of $X$. Since $A_j$ are disjoint and $h$ is a well defined $S-$map, $X_j$ are also disjoint and $X=\dot\bigcup_{j \in J}X_j$. Clearly $\im(h|_{X_j})\subseteq A_j$ for each $j \in J$. If $h$ is an epimorphism then none of the $X_i$ are empty and so $J=I$.
\end{proof}

\begin{proposition}
Let $S$ be a monoid and let $\mathcal{X}$ satisfy the property that  for each $i \in I$, $\dot\bigcup_{i \in I}X_i \in \mathcal{X} \Leftrightarrow X_i \in \mathcal{X}$. Then each $A_i$ has an $\mathcal{X}$-precover if and only if  $\dot\bigcup_{i \in I}A_i$ has an $\mathcal{X}$-precover.
\end{proposition}

\begin{proof}
For each $i\in I$, let $g_i:C_i \rightarrow A_i$ be an $\X$-precover of $A_i$. Define $g:\dot\bigcup_{i \in I}C_i \rightarrow \dot\bigcup_{i \in I}A_i$ to be the obvious induced map where $g|_{C_i} = g_i$ for each $i \in I$. We claim this is an $\X$-precover of $\dot\bigcup_{i \in I}A_i$. Let $X \in \X$ and let $h: X \rightarrow \dot\bigcup_{i \in I}A_i$. By Lemma \ref{disjoint-hom}, there is a subset $J \subseteq I$ such that $X=\dot\bigcup_{j \in J}X_j$ and $\im(h|_{X_j}) \subseteq A_j$ for each $j \in J$. Now by the hypothesis $X_j \in \X$ so since $C_j$ is an $\X$-precover of $A_j$, there exists $f_j \in $ Hom$_S(X_j,C_j)$ such that $h|_{X_j}=g_j f_j$. So define $f:\dot\bigcup_{j \in J}X_j \rightarrow \dot\bigcup_{i \in I}C_i$ to be the obvious induced map with $f|_{X_j}=f_j$ for each $j \in J$ and clearly $gf=h$.

\smallskip

Conversely let $g:C \rightarrow \dot\bigcup_{i \in I}A_i=A$ be an $\X$-precover of $A$. Let $i\in I$ and define $C_i=\{c \in C : g(c) \in A_i\}$, and let $g_i=g|_{C_i}$. Suppose that $X$ is an $S-$act and suppose that $h \in$ Hom$_S(X,A_i)$. Then clearly $h \in$ Hom$_S(X,A)$ and so by the $\X$-precover property there exists an $f \in$ Hom$_S(X,C)$ such that $h=gf$. In fact $g(f(X)) =h(X) \subseteq A_i$ and so $f \in$ Hom$_S(X,C_i)$ and $h_i=g_if$. By the hypothesis, $C_i \in \X$ and hence $g_i:C_i \rightarrow A_i$ is an $\X$-precover of $A_i$.
\end{proof}

By Lemma~\ref{coproduct-decompose}, the classes $\Pr, \SF, \CP$ and $\F$  all satisfy this property and so for any of these classes, to show that all $S-$acts have $\X$-precovers it is enough to show that the indecomposable $S-$acts have $\X$-precovers.

\begin{lemma} \label{no-map}
Let $S$ be a monoid. The one element $S-$act $\Theta_S$ has an $\X$-precover if and only if there exists an $S-$act $A \in \X$ such that Hom$_S(X,A)\neq\emptyset$ for all $X \in \X$.
\end{lemma}

\begin{proof}
Let $\Theta_S=\{\theta\}$, let $A \in \X$ and let $g:A \rightarrow \Theta_S$ be given by $g(a) = \theta$. Given any $S-$act $X \in \X$ with $S-$map $h:X \rightarrow \Theta_S$, clearly $gf=h$ for every $f \in $ Hom$_S(X,A)$.
\end{proof}

We now show that colimits of $\X-$precovers are $\X-$precovers. To be more precise
\begin{lemma}\label{direct-limit-covers-lemma}
Let $S$ be a monoid, let $\X$ be a class of $S-$acts closed under colimits and let $A$ be an $S-$act. Suppose that $(X_i,\phi_{i,j})$ is a direct system of $S-$acts with $X_i\in\X$ for each $i\in I$ and with colimit $(X,\alpha_i)$. Suppose also that for each $i\in I$ $f_i:X_i\to A$ is an $\X-$precover of $A$ such that for all $i\le j$, $f_j\phi_{i,j}=f_i$. Then there exists an $\X-$precover $f:X\to A$ such that $f\alpha_i = f_i$ for all $i\in I$.
\end{lemma}

\begin{proof}
We have a commutative diagram
$$
\begin{tikzpicture}[description/.style={fill=white,inner sep=2pt}]
\matrix (m) [matrix of math nodes, row sep=3em,
column sep=2.5em, text height=1.5ex, text depth=0.25ex]
{X_i& &X_j \\
 & X &\\ 
& A & \\};
\path[->,font=\scriptsize]
(m-1-1) edge node[auto, right] {$\alpha_i$} (m-2-2)
(m-1-3) edge node[auto, left] {$\alpha_j$} (m-2-2)
(m-1-1) edge node[auto, left] {$f_i$} (m-3-2)
(m-1-3) edge node[auto, right] {$f_j$} (m-3-2)
(m-1-1) edge[->] node[auto,above] {$\phi_{i,j}$} (m-1-3);
\end{tikzpicture}
$$
and so there exists a unique $S-$map $f:X\to A$ such that $f\alpha_i = f_i$ for all $i\in I$. If $F\in \X$ and if $g:F\to A$ then for each $i\in I$ there exists $h_i:F\to X_i$ such that $f_ih_i=g$. Choose any $i\in I$ and let $h:F\to X$ be given by $h=\alpha_ih_i$. Then $fh=g$ as required.
\end{proof}

The motivation for the next few results comes mainly from~\cite{xu-96}.

\begin{lemma}\label{pre-cover-lemma1}
Let $S$ be a monoid and $\X$ a class of $S-$acts closed under directed colimits. Let $A$ be an $S-$act and suppose that $k:C\to A$ is an $\X-$precover of $A$. Then there exists an $\X-$precover $\bar k:\bar C\to A$ and an $S-$map $g:C\to\bar C$ with $\bar k g = k$ such that for any $\X-$precover $k^\ast:C^\ast\to A$ and any $S-$map $h:\bar C\to C^\ast$ with $k^\ast h = \bar k$ then $h|_{\im(g)}$ is a monomorphism.
\end{lemma}
\begin{proof}
Suppose, by way of contradiction, that for all $\X-$precovers $\bar k:\bar C\to A$ and $S-$maps $g:C\to\bar C$ with $\bar k g=k$ there exists an $\X-$precover $k^\ast:C^\ast\to A$ and an $S-$map $h:\bar C\to C^\ast$ with $k^\ast h=\bar k$ and such that $h|_{\im(g)}$ is not a monomorphism. So in particular when $\bar C = C, \bar k = k$ and $g=1_C$ then there exists an $\X-$precover $k_1:C_1\to A$ and an $S-$map $g_{1,0}:C\to C_1$ with $k_1 g_{1,0}=k$ and such that $g_{1,0}|_{\im(1_C)}$ is not a monomorphism.

Now let $\kappa\ge2$ be an ordinal and suppose that for all ordinals $\alpha<\kappa$ there is an $\X-$precover $k_\alpha:C_\alpha\to A$ and $S-$maps $g_{\alpha,\beta}:C_\beta\to C_\alpha$ for $\beta<\alpha$ such that for any triple $\gamma<\delta<\alpha, g_{\alpha,\gamma}=g_{\alpha,\delta}g_{\delta,\gamma}$ and
$$
\ker(g_{1,0})\subsetneq\ldots\subsetneq\ker{(g_{\alpha,0})}\subsetneq\ldots \subseteq C\times C.
$$
We proceed by transfinite induction. First, if $\kappa$ is not a limit ordinal then on putting $\bar C=C_{\kappa-1}, \bar k=k_{\kappa-1}$ and $g=g_{\kappa-1,0}$ we deduce there exists an $\X-$precover $k_\kappa:C_\kappa\to A$ and an $S-$map $g_{\kappa,\kappa-1}:C_{\kappa-1}\to C_\kappa$ with $k_\kappa g_{\kappa,\kappa-1}=g_{\kappa-1,0}$ such that $g_{\kappa,\kappa-1}|_{\im(g_{\kappa-1,0})}$ is not a monomorphism. For $\beta<\kappa-1$ let $g_{\kappa,\beta}=g_{\kappa,\kappa-1}g_{\kappa-1,\beta}$ so that  $\ker(g_{\kappa-1,0})\subsetneq\ker(g_{\kappa,0})$ and for  $\gamma<\delta<\kappa, g_{\kappa,\gamma}=g_{\kappa,\delta}g_{\delta,\gamma}$ as required.

\smallskip

Now if $\kappa$ is a limit ordinal then let $(C_\kappa,g_{\kappa,\alpha}:C_\alpha\to C_\kappa)$ be the directed colimit of the system $(C_\alpha,g_{\alpha,\beta})$

and consider the diagram
$$
\begin{tikzpicture}[description/.style={fill=white,inner sep=2pt}]
\matrix (m) [matrix of math nodes, row sep=3em,
column sep=2.5em, text height=1.5ex, text depth=0.25ex]
{C_\beta & &C_\alpha \\
 & C_\kappa &\\ 
& A & \\};
\path[->,font=\scriptsize]
(m-1-1) edge node[auto, right] {$g_{\kappa,\beta}$} (m-2-2)
(m-1-3) edge node[auto, left] {$g_{\kappa,\alpha}$} (m-2-2)
(m-1-1) edge node[auto, left] {$k_\beta$} (m-3-2)
(m-1-3) edge node[auto, right] {$k_\alpha$} (m-3-2)
(m-2-2) edge node[auto, above right] {$k_\kappa$} (m-3-2)
(m-1-1) edge[->] node[auto,above] {$g_{\alpha,\beta}$} (m-1-3);
\end{tikzpicture}
$$
where $k_\kappa:C_\kappa\to A$ is the unique $S-$map that makes the diagram commutative. Then by Lemma~\ref{direct-limit-covers-lemma} we deduce that $k_\kappa:C_\kappa\to A$ is an $\X-$precover for $A$. In addition we see that for  $\gamma<\delta<\kappa, g_{\kappa,\gamma}=g_{\kappa,\delta}g_{\delta,\gamma}$ and that $\ker(g_{\delta,0})\subseteq\ker(g_{\kappa,0})$. But $\ker(g_{\delta,0})\subsetneq\ker(g_{\delta+1,0})\subseteq\ker(g_{\kappa,0})$ and so $\ker(g_{\delta,0})\subsetneq\ker(g_{\kappa,0})$ as required.

\smallskip

It then follows that $|C\times C|$ is greater than the cardinality of every ordinal which is a clear contradiction.
\end{proof}

\begin{lemma}\label{pre-cover-lemma2}
Let $S$ be a monoid and $\X$ a class of $S-$acts closed under directed colimits. Let $A$ be an $S-$act and suppose that $k:C\to A$ is an $\X-$precover of $A$. Then there exists an $\X-$precover $\bar k:\bar C\to A$ such that for any $\X-$precover $k^\ast:C^\ast\to A$ and any $S-$map $h:\bar C\to C^\ast$ with $k^\ast h = \bar k$ then $h$ is a monomorphism.
\end{lemma}
\begin{proof}
By Lemma~\ref{pre-cover-lemma1} there exists an $\X-$precover $k_1:C_1\to A$ and an $S-$map $g_{1,0}:C\to C_1$ with $k_1 g_{1,0}=k$ such that for any $\X-$precover $k^\ast:C^\ast\to A$ and any $S-$map $h:C_1\to C^\ast$ with $k^\ast h=k_1$ then $h|_{\im(g_{1,0})}$ is a monomorphism.
Now let $n>1$ and suppose by way of induction that there is an $\X-$precover $k_{n-1}:C_{n-1}\to A$ and a map $g_{n-1,n-2}:C_{n-2}\to C_{n-1}$ with $k_{n-1} g_{n-1,n-2}=k_{n-2}$ and such that for any $\X-$precover $k^\ast:C^\ast\to A$ and any $S-$map $h:C_{n-1}\to C^\ast$ with $k^\ast h = k_{n-1}$ then $h|_{\im(g_{n-1,n-2})}$ is a monomorphism (here we obviously assume $C_0=C$ and $k_0=k$).
$$
\begin{tikzpicture}[description/.style={fill=white,inner sep=2pt}]
\matrix (m) [matrix of math nodes, row sep=3em,
column sep=2.5em, text height=1.5ex, text depth=0.25ex]
{C_{n-2} && C_{n-1} && A\\ 
 & &&&C^\ast\\};
\path[->,font=\scriptsize]
(m-1-3) edge node[auto,below left] {$h$} (m-2-5)
(m-2-5) edge node[auto, right] {$k^\ast$} (m-1-5)
(m-1-1) edge[->] node[auto,above] {$g_{n-1,n-2}$} (m-1-3)
(m-1-3) edge[->] node[auto,above] {$k_{n-1}$} (m-1-5);
\end{tikzpicture}
$$

Then by Lemma~\ref{pre-cover-lemma1} we deduce that there exists an $\X-$precover $k_n:C_n\to A$ and a map $g_{n,n-1}:C_{n-1}\to C_n$ with $k_n g_{n,n-1}=k_{n-1}$ and such that for any $\X-$precover $k^\ast:C^\ast\to A$ and any $S-$map $h:C_n\to C^\ast$ with $k^\ast h = k_n$ then $h|_{\im(g_{n,n-1})}$ is a monomorphism. 

\smallskip

Now let $(C_\omega,g_{\omega,n}:C_n\to C_\omega)$ be the directed colimit of the system $(C_n,g_{n,n-1})$ and consider the diagram
$$
\begin{tikzpicture}[description/.style={fill=white,inner sep=2pt}]
\matrix (m) [matrix of math nodes, row sep=3em,
column sep=2.5em, text height=1.5ex, text depth=0.25ex]
{C_{n-1} & &C_n \\
 & C_\omega &\\ 
& A & \\};
\path[->,font=\scriptsize]
(m-1-1) edge node[auto, right] {$g_{\omega,n-1}$} (m-2-2)
(m-1-3) edge node[auto, left] {$g_{\omega,n}$} (m-2-2)
(m-1-1) edge node[auto, left] {$k_{n-1}$} (m-3-2)
(m-1-3) edge node[auto, right] {$k_n$} (m-3-2)
(m-2-2) edge node[auto, above right] {$k_\omega$} (m-3-2)
(m-1-1) edge[->] node[auto,above] {$g_{n,n-1}$} (m-1-3);
\end{tikzpicture}
$$
where $k_\omega:C_\omega\to A$ is the unique $S-$map that makes the diagram commutative. Then by Lemma~\ref{direct-limit-covers-lemma} we deduce that $k_\omega:C_\omega\to A$ is an $\X-$precover for $A$. We claim that this $\X-$precover has the desired properties. So let $k^\ast:C^\ast\to A$ be an $\X-$precover of $A$ and let $h:C_\omega\to C^\ast$ be an $S-$map  with $k^\ast h = k_\omega$. Suppose also that $h(x) = h(y)$ for $x,y \in C_\omega$. Then there exists $m,n>0$ and $x_m\in C_m, y_n\in C_n$ such that $g_{\omega,m}(x_m) = x$ and $g_{\omega,n}(y_m)=y$. Assume without loss of generality that $m\le n$ and let $z_n=g_{n,m}(x_m)$. Then $$
hg_{\omega,n+1}(g_{n+1,n}(z_n)) = hg_{\omega,n}(z_n) = hg_{\omega,n}(y_n) = hg_{\omega,n+1}(g_{n+1,n}(y_n)).
$$
But $hg_{\omega,n+1}:C_{n+1}\to C^\ast$ and $hg_{\omega,n+1}|_{\im(g_{n+1,n})}$ is therfore a monomorphism. Hence $g_{n+1,n}(z_n)=g_{n+1,n}(y_n)$ and so
$$
x=g_{\omega,m}(x_m)  = g_{\omega,n+1}(g_{n+1,n}(z_n)) = g_{\omega,n+1}(g_{n+1,n}(y_n)) = g_{\omega,n}(y_n) = y
$$
as required.
\end{proof}

We can now deduce one of our main theorems.

\begin{theorem}
Let $S$ be a monoid, let $A$ be an $S-$act and let $\X$ be a class of $S-$acts closed under directed colimits. If $A$ has an $\X-$precover then $A$ has an $\X-$cover.
\end{theorem}
\begin{proof}
By Lemma~\ref{pre-cover-lemma2} there exists an $\X-$precover $k_0:C_0\to A$ such that for any $\X-$precover $k^\ast:C^\ast\to A$ and any $S-$map $h:C_0\to C^\ast$ with $k^\ast h = k_0$ then $h$ is a monomorphism.  We show that $k_0:C_0\to A$ is in fact an $\X-$cover of $A$.

\smallskip

Assume by way of contradiction that $A$ does not have an $\X-$cover. Let $C_1=C_0$ and $k_1=k_0$. Then there exists $g_{1,0}:C_0\to C_1$ with $k_1 g_{1,0}=k_0$ and such that $g_{1,0}$ is a monomorphism but not an epimorphism. It follows that
$$
\im(g_{1,0})\subsetneq C_1=C_0.
$$
By way of transfinite induction suppose that $\kappa\ge2$ is an ordinal such that for all ordinals $\alpha<\kappa$ there exists an $\X$-precover $k_\alpha:C_\alpha\to A$ such that
\begin{enumerate}
\item[(1)] for any $\X-$precover $k^\ast:C^\ast\to A$ and any $S-$map $h:C_\alpha\to C^\ast$ with $k^\ast h = k_\alpha$ then $h$ is a monomorphism;
\item[(2)] for all ordinals $\beta<\alpha$ there exists $S-$maps $g_{\alpha,\beta}:C_\beta\to C_\alpha$ which are monomorphisms but not epimorphisms and $\im(g_{\alpha,\beta})\subsetneq C_\alpha$;
\item[(3)] for all ordinals $\gamma<\beta<\alpha$, $g_{\alpha,\gamma}=g_{\alpha,\beta}g_{\beta,\gamma}$ and
$$
\im(g_{\alpha,\gamma})\subsetneq\im(g_{\alpha,\beta}).
$$
\end{enumerate}
We show that $\kappa$ also possesses these properties. If $\kappa$ is not a limit ordinal then let $C_\kappa=C_{\kappa-1}$ and $k_\kappa=k_{\kappa-1}$. Then clearly $k_\kappa:C_\kappa\to A$ satisfies the condition of (1) above. Also there exists $g_{\kappa,\kappa-1}:C_{\kappa-1}\to C_\kappa$ with $k_\kappa g_{\kappa,\kappa-1}=k_{\kappa-1}$ which is a monomorphism but not an epimorphism. For each $\beta<\kappa$ let $g_{\kappa,\beta}=g_{\kappa,\kappa-1}g_{\kappa-1,\beta}$. Then since $g_{\kappa,\kappa-1}$ is not onto it follows that $g_{\kappa,\beta}$ is not an epimorphism but it is a monomorphism and so $\im(g_{\kappa,\beta})\subsetneq C_\kappa$. By the inductive hypothesis, if $\gamma<\beta<\kappa$, $g_{\kappa,\gamma}=g_{\kappa,\beta}g_{\beta,\gamma}$ and in addition $\im(g_{\kappa,\gamma})\subsetneq\im(g_{\kappa,\beta})$.

Now suppose that $\kappa$ is a limit ordinal and let $(C_\kappa,g_{\kappa,\beta}:C_\beta\to C_\kappa)$ be the directed colimit of the system $(C_\beta,g_{\beta,\gamma})$ and consider the diagram
$$
\begin{tikzpicture}[description/.style={fill=white,inner sep=2pt}]
\matrix (m) [matrix of math nodes, row sep=3em,
column sep=2.5em, text height=1.5ex, text depth=0.25ex]
{C_\gamma & &C_\beta \\
 & C_\kappa &\\ 
& A & \\};
\path[->,font=\scriptsize]
(m-1-1) edge node[auto, right] {$g_{\kappa,\gamma}$} (m-2-2)
(m-1-3) edge node[auto, left] {$g_{\kappa,\beta}$} (m-2-2)
(m-1-1) edge node[auto, left] {$k_\gamma$} (m-3-2)
(m-1-3) edge node[auto, right] {$k_\beta$} (m-3-2)
(m-2-2) edge node[auto, above right] {$k_\kappa$} (m-3-2)
(m-1-1) edge[->] node[auto,above] {$g_{\beta,\gamma}$} (m-1-3);
\end{tikzpicture}
$$
where $k_\kappa:C_\kappa\to A$ is the unique $S-$map that makes the diagram commute. Then by Lemma~\ref{direct-limit-covers-lemma} we deduce that that $k_\kappa:C_\kappa\to A$ is an $\X-$precover for $A$. In addition we see that for  $\gamma<\beta<\kappa, g_{\kappa,\gamma}=g_{\kappa,\beta}g_{\beta,\gamma}$ and that since each $g_{\beta,\gamma}$ is a monomorphism then so is each $g_{\kappa,\beta}$. Suppose that $g_{\kappa,\gamma}$ in onto for some $\gamma<\kappa$. Then for each $\gamma<\beta<\kappa$, since $g_{\kappa,\beta}$ is a monomorphism, it follows that $g_{\beta,\gamma}$ is also onto which is a contradiction and so $g_{\kappa,\gamma}$ is not an epimorphism for any $\gamma<\kappa$. It is then clear that
$$
\im(g_{\kappa,\gamma})\subsetneq\im(g_{\kappa,\beta})\subsetneq C_\kappa.
$$
Finally let $k^\ast:C^\ast\to A$ be an $\X-$precover and let $h:C_\kappa\to C^\ast$ be such that $k^\ast h=k_\kappa$. Then for each $\beta<\kappa$ we have a commutative diagram
$$
\begin{tikzpicture}[description/.style={fill=white,inner sep=2pt}]
\matrix (m) [matrix of math nodes, row sep=3em,
column sep=2.5em, text height=1.5ex, text depth=0.25ex]
{C_\beta & C_\kappa&A \\
& & C^\ast\\};
\path[->,font=\scriptsize]
(m-2-3) edge node[auto, right] {$k^\ast$} (m-1-3)
(m-1-2) edge node[auto, right] {$h$} (m-2-3)
(m-1-1) edge node[auto, left] {$hg_{\kappa,\beta}$} (m-2-3)
(m-1-2) edge node[auto, above] {$k_\kappa$} (m-1-3)
(m-1-1) edge[->] node[auto,above] {$g_{\kappa,\beta}$} (m-1-2);
\end{tikzpicture}
$$
and by assumption $hg_{\kappa,\beta}$ is a monomorphism. Hence by Lemma~\ref{directlimit-monomorphism-lemma} it follows that $h$ is a monomorphism. In particular we can deduce that there is a monomorphism $C_\kappa\to C$.

Consequently we see that for any ordinal $\kappa$ we have a chain of length $\kappa$
$$
\im(g_{\kappa,0})\subsetneq\ldots\subsetneq\im(g_{\kappa,\beta})\subsetneq\ldots\subsetneq C_\kappa\subseteq C
$$
which is a contradiction.
\end{proof}

\bigskip

It is clear that a necessary condition for an $S-$act $A$ to have an $\X-$precover is that there exists $X\in\X$ with  $\text{Hom}_S(X,A)\ne\emptyset$. This condition is always satisfied in the category of modules over a ring (or indeed any category with a zero object), as every Hom-set is always non-empty, but this is not always the case for $S-$acts.

\smallskip

Let $S$ be a monoid and let $\X$ be a class of $S-$acts. We say that $\X$ satisfies the {\em (weak) solution set condition} if for all $S-$acts $A$ there exists a set $S_A\subseteq\X$ such that for all (indecomposable) $X\in\X$ and all $S-$maps $h:X\to A$ there exists $Y\in S_A$, $f:X\to Y$ and $g:Y\to A$ such that $h=gf$.

\begin{theorem}\label{solution-set-theorem}
Let $S$ be a monoid and let $\X$ be a class of $S-$acts such that
 $\dot\bigcup_{i \in I}X_i \in \X \Leftrightarrow X_i \in \X$ for each $i \in I$.
Then every $S-$act has an $\X-$precover if and only if
\begin{enumerate}
\item for every $S-$act $A$ there exists an $X$ in $\X$ such that $\text{Hom}_S(X,A)\ne\emptyset$;
\item $\X$ satisfies the weak solution set condition;
\end{enumerate}
\end{theorem}

\begin{proof}
Suppose that $\X$ satisfies the given conditions. Let $A$ be a $S-$act and let $S_A=\{C_i:i\in I\}$ be as given in the weak solution set condition. Notice that by property (1) $S_A\ne\emptyset$. Moreover we can assume that for all $Y\in S_A$, $\text{Hom}_S(Y,A)\ne\emptyset$ as $S_A\setminus\{Y\in S_A|\text{Hom}_S(Y,A)=\emptyset\}$ will also satisfy the requirements of the solution set condition.

For each $i\in I$ and for each $S-$map $g:C_i\to A$ let $C_{i,g}$ be an isomorphic copy of $C_i$ with isomorphism $\phi_{i,g}:C_{i,g} \to C_i$ (recall that we are assuming that $\X$ is closed under isomorphisms). Let
$$
C_A=\dot\bigcup_{i\in I,g \in \text{Hom}_S(C_i,A)}C_{i,g}.
$$
By hypothesis, $C_A \in \X$ and we can define an $S-$map $\bar g:C_A \rightarrow A$ by $\bar g|_{C_{i,g}}=g\phi_{i,g}$ for each $i \in I$, $g \in $ Hom$_S(C_i,A)$. We claim that $(C_A,\bar g)$ is an $\X-$precover for $A$. Let $X \in \X$ and let $h:X \rightarrow A$ be an $S-$map. By the hypothesis $X=\dot\bigcup_{j \in J}X_j$ is a coproduct of indecomposable $S-$acts with $X_j\in\X$ for each $j\in J$. Further, by the hypothesis, there exists $C_{i_j}\in S_A$, $f_j:X_j\to C_{i_j}$ and $g_j:C_{i_j}\to A$ such that $g_jf_j = h|_{X_j}$. Now $\bar g|_{C_{i_j,g_j}}\phi_{i_j,g_j}^{-1} = g_j$ and so both triangles and the outer square in the following diagram commute (where the unlabelled arrows are the obvious inclusion maps).
$$
\begin{tikzpicture}[description/.style={fill=white,inner sep=2pt}]
\matrix (m) [matrix of math nodes, row sep=3em,
column sep=2.5em, text height=1.5ex, text depth=0.25ex]
{C_A & A \\
C_{i_j,g_j}& X=\dot\bigcup X_j\\
C_{i_j}&X_j\\};
\path[->,font=\scriptsize]
(m-3-2) edge node[auto,below] {$f_j$} (m-3-1)
(m-3-1) edge node[auto,left] {$g_j$} (m-1-2)
(m-3-2) edge node[auto, right] {$$} (m-2-2)
(m-2-2) edge node[auto, right] {$h$} (m-1-2)
(m-2-1) edge node[auto, left] {} (m-1-1)
(m-3-1) edge node[auto, left] {$\phi_{i_j,g_j}^{-1}$} (m-2-1)
(m-1-1) edge[->] node[auto,above] {$\bar g$} (m-1-2);
\end{tikzpicture}
$$
So define $f:X\to C_A$ by $f|_{X_j}=\phi_{i_j,g_j}^{-1}f_j$ and note that $\bar gf = h$ as required.

\medskip
Conversely if $A$ is an $S-$act with an $\X-$precover $C_A$, then $\text{Hom}_S(C_A,A)\ne\emptyset$ and on putting $S_A=\{C_A\}$ we see that $\X$ satisfies the (weak) solution set condition.
\end{proof}

\smallskip
Note from the proof of Theorem~\ref{solution-set-theorem} that we can also deduce
\begin{theorem}
Let $S$ be a monoid and let $\X$ be a class of $S-$acts such that
$X_i \in \X$ for each $i \in I\Rightarrow\dot\bigcup_{i \in I}X_i \in \X$.
Then every $S-$act has an $\X-$precover if and only if
\begin{enumerate}
\item for every $S-$act $A$ there exists an $X$ in $\X$ such that $\text{Hom}_S(X,A)\ne\emptyset$;
\item $\X$ satisfies the solution set condition;
\end{enumerate}
\end{theorem}
\smallskip

\begin{corollary}\label{lambda-skeleton-corollary}
Let $S$ be a monoid and let $\X$ be a class of $S-$acts such that
\begin{enumerate}
\item $\dot\bigcup_{i \in I}X_i \in \X \Leftrightarrow X_i \in \X$ for each $i \in I$;
\item for every $S-$act $A$ there exists an $X$ in $\X$ such that $\text{Hom}_S(X,A)\ne\emptyset$;
\item there exists a cardinal $\lambda$ such that for every indecomposable $X$ in $\X$, $|X|<\lambda$.
\end{enumerate}
Then every $S-$act has an $\X-$precover.
\end{corollary}

\begin{proof}
By (3) there exists a $\lambda-$skeleton $C=\{C_i  : i \in I\}$, for the indecomposable $S-$acts in $\X$. Suppose that $A$ is an $S-$act  and let $S_A=C$. If $X \in \X$ is indecomposable and if $h:X \to A$ is an $S-$map then there exists an isomorphism $\phi:X \to C_i$ for some $C_i \in C$ and we have an $S-$map $h\phi^{-1}:C_i \to A$ and clearly $h=h\phi^{-1}\phi$ and so $\X$ satisfies the weak solution set condition.
\end{proof}

\medskip

Let $A$ be an $S-$act and let $\rho$ be a congruence on $A$. We shall say that $\rho$ is {\em $\X-$pure} if $A/\rho \in\X$.
The inspiration for some of the following results comes from~\cite{xu-96}.

\begin{theorem}
Let $S$ be a monoid, let $\X$ be a class of $S-$acts and suppose that $A$ is an $S-$act such that $\psi:F\to A$ is an $\X-$precover. Suppose also that the set of $\X-$pure congruences on $F$ is closed under unions of chains. Then there exists an $\X-$precover $\phi:G\to A$ of $A$ such that there is no non-identity $\X-$pure congruence $\rho\subset\ker(\phi)$ on $G$.
\end{theorem}
\begin{proof}
First, if there does not exists a non-identity $\X-$pure congruence $\sigma\subseteq\ker(\psi)$ on $F$ then we let $G=F$ and $\phi=\psi$. Otherwise by assumption any chain of $\X-$pure congruences on $F$ contained in $\ker(\psi)$ has an upper bound and so by Zorn's lemma there is a maximum $\sigma$ say. Let $G=F/\sigma$ and let $\phi:G\to A$ by the natural map which makes
$$
\begin{tikzpicture}[description/.style={fill=white,inner sep=2pt}]
\matrix (m) [matrix of math nodes, row sep=3em,
column sep=2.5em, text height=1.5ex, text depth=0.25ex]
{F&A \\
 G&\\};
\path[->,font=\scriptsize]
(m-1-1) edge node[auto, above] {$\psi$} (m-1-2)
(m-1-1) edge node[auto, left] {$\sigma^\natural$} (m-2-1)
(m-2-1) edge[->] node[auto,right] {$\phi$} (m-1-2);
\end{tikzpicture}
$$
commute. Then it is easy to check that $\phi:G\to A$ is an $\X-$precover as if $H\in\X$ and if $f:H\to A$ then there exists $g:H\to F$ such that $\psi g=f$. So $\sigma^\natural g:H\to G$ and $\phi\sigma^\natural g = \psi g = f$ and $\phi:G\to A$ is an $\X-$precover.

Finally suppose that $\rho$ is an $\X-$pure congruence on $G$ such that $\rho\subset\ker(\phi)$. Then by Remark~\ref{remark-1}, $\sigma/\rho$ is an $\X-$pure congruence on $F$ containing $\sigma$ and $\sigma/\rho=\ker(\rho^\natural\sigma^\natural)\subseteq\ker(\psi)$. By the maximality of $\sigma$ it follows that $\sigma=\sigma/\rho$ and so $\rho=1_G$, a contradiction as required.
\end{proof}

Following~\cite{bican-09} we can extend this result as follows.

\begin{proposition}
Let $S$ be a monoid and let $\X$ be a class of $S-$acts. If $A$ is an $S-$act such that $\psi:F\to A$ is an $\X-$cover then there is no non-idenity $\X-$pure congruence $\rho\subset\ker\psi$ on $F$.
\end{proposition}
\begin{proof}
Let $\rho\subset\ker\psi$ be an $\X-$pure congruence on $F$. Then there is an induced $S-$map $\phi:F/\rho\to A$ such that $\phi\rho^\natural=\psi$. Since $(F,\psi)$ is a precover then there exists an $S-$map $\theta : F/\rho\to F$ such that $\psi\theta=\phi$.
$$
\begin{tikzpicture}[description/.style={fill=white,inner sep=2pt}]
\matrix (m) [matrix of math nodes, row sep=3em,
column sep=2.5em, text height=1.5ex, text depth=0.25ex]
{F&A \\
 &F/\rho\\};
\path[->,font=\scriptsize]
(m-1-1) edge node[auto, above] {$\psi$} (m-1-2)
(m-2-2) edge node[auto, left] {$\theta$} (m-1-1)
(m-2-2) edge[->] node[auto,right] {$\phi$} (m-1-2);
\end{tikzpicture}
$$
Hence $\psi\theta\rho^\natural=\phi\rho^\natural=\psi$ and so $\theta\rho^\natural$ is an automorphism of $F$. Hence $\rho^\natural$ is a monomorphism and so $\rho=1_A$ as required.
\end{proof}

Let $\X$ be a class of $S-$acts. Let us say that $\X$ is {\em (weakly) congruence pure} if for each cardinal $\lambda$ 
there exists a cardinal $\kappa >\lambda$ such that for every (indecomposable) $X\in \X$ with $|X|\ge\kappa$ and every congruence $\rho$ on $X$ with $|X/\rho|\le\lambda$ there exists an $\X-$pure congruence $1_X\ne\theta\subseteq\rho$ of $X$.

\begin{theorem}\label{congruence-pure-theorem}
Let $S$ be a monoid, let  $\X$ be a class of $S-$acts such that
\begin{enumerate}
\item $\dot\bigcup_{i \in I}X_i \in \X \Leftrightarrow X_i \in \X$ for each $i \in I$;
\item for every $S-$act $A$ there exists an $X$ in $\X$ such that $\text{Hom}_S(X,A)\ne\emptyset$;
\item for each $X\in\X$ the set of all $\X-$pure congruences on $X$ is closed under unions of chains;
\item $\X$ is weakly congruence pure.
\end{enumerate}
Then $\X$ satisfies the weak solution set condition and so every $S-$act has an $\X-$precover.
\end{theorem}

\begin{proof}
Let $A$ be an $S-$act, let $\lambda = \max\{|A|,\aleph_0\}$, let $\kappa$ be as given in the weakly congruence pure condition and let $S_A$ be any $\kappa-$skeleton of $\X$ consisting of $S-$acts of cardinalities less than $\kappa$, .
Suppose that $X$ is an indecomposable $S-$act and that $h:X\to A$ is an $S-$map. If $|X|<\kappa$ then let $Y\in S_A$ be an isomorphic copy of $X$ and let $f:X\to Y$ be an isomorphism and define $g:Y\to A$ by $g=hf^{-1}$ so that $h=gf$.

Suppose now that $|X|\ge\kappa$. Then $|X/\ker(h)|=|\im(h)|\le\lambda$ and so there exists an $\X-$pure congruence $1_{X}\ne\theta\subseteq\ker(h)$ on $X$ with  $X/\theta\in\X$. In fact, using a combination of Zorn's lemma and the hypothesis that the set of $\X-$pure congruences on $X$ is closed under unions of chains, we can assume that $\theta$ is maximal with respect to this property. Now let $\bar h : X/\theta\to A$ be the unique map such that
$$
\begin{tikzpicture}[description/.style={fill=white,inner sep=2pt}]
\matrix (m) [matrix of math nodes, row sep=3em,
column sep=2.5em, text height=1.5ex, text depth=0.25ex]
{X & X/\theta \\
 & A\\ };
\path[->,font=\scriptsize]
(m-1-1) edge node[auto, left] {$h$} (m-2-2)
(m-1-2) edge node[auto,right] {$\bar h$} (m-2-2)
(m-1-1) edge[->] node[auto,above] {$\theta^{\natural}$} (m-1-2);
\end{tikzpicture}
$$
commutes. Notice that since $\im(\bar h) = \im(h)$ then
$$
|(X/\theta)/\ker(\bar h)| = |X/\ker(h)|\le\lambda.
$$
Now suppose, by way of contradiction, that $1_{X/\theta}\ne\rho\subseteq\ker(\bar h)$ is an $\X-$pure congruence on $X/\theta$ so that $(X/\theta)/\rho\in\X$. Then by Remark~\ref{remark-1} and since $X\in\X$ it follows that $\theta/\rho$ is an $\X-$pure congruence on $X$ containing $\theta$ and since $\rho\subseteq\ker(\bar h)$ it easily follows that $\theta/\rho\subseteq\ker(h)$. Hence by the maximality of $\theta$ we deduce that $\theta/\rho=\theta$ and so $\rho=1_{X/\theta}$. Therefore it follows that $X/\theta$ does not contain a non-identity $\X-$pure congruence contained in $\ker(\bar h)$ and since by Lemma~\ref{indecomposable-epi-lemma}, $X/\theta$ is indecomposable and since $\X$ is weakly congruence pure we deduce that $|X/\theta|<\kappa$.
Consequently it follows that there exists $Y\in S_A$ and an isomorphism $\bar f:X/\theta\to Y$ and so define $f:X\to Y$ by $f=\bar f\theta^\natural$ and $g:Y\to A$ by $g=\bar h{\bar f}^{-1}$ so that $gf=h$.

Hence $\X$ satisfies the weak solution set condition and the result follows from Theorem~\ref{solution-set-theorem}.
\end{proof}
A similar condition to this is considered in \cite{bican-01} and forms the basis of one of the proofs of the flat cover conjecture.

\section{Strongly Flat and Condition $(P)$ Covers}

In this section we apply some of the previous results to the specific classes $\X=\SF$ and $\X=\CP$. In particular note from Lemma \ref{coproduct-decompose} that $\dot\bigcup_{i \in I}X_i \in \X \Leftrightarrow X_i \in \X$ for each $i \in I$ holds for both $\X=\SF$ and $\X=\CP$. Also, since $S$ is strongly flat (and hence satisfies condition $(P)$) given any $S-$act $A$ there exists an $X$ in $\X$ such that $\text{Hom}_S(X,A)\ne\emptyset$.

Let $A$ be an $S-$act and let $\rho$ be a congruence on $A$. Recall that we say that $\rho$ is {\em $\X-$pure} if $A/\rho \in\X$.
So, by Propositions \ref{ce-pure-proposition} and \ref{cp-pure-proposition}, Corollary~\ref{x-pure-corollary} and \cite[Corollary 4.1.3 and Theorem 4.1.4]{ahsan-08} we deduce

\begin{corollary}
Let $S$ be a monoid, let $X$ be an $S-$act and let $\rho$ be a congruence on $X$.
\begin{enumerate}
\item If $X\in\E$ then $\rho$ is $\E$-pure if and only if it is $1-$pure.
\item If $X\in\CP$ then $\rho$ is $\CP$-pure if it is $2-$pure.
\item If $X\in\SF$ then $\rho$ is $\SF$-pure if and only if it is pure if and only if it is $2-$pure.
\item If $X\in\Pr$ then $\rho$ is $\Pr$-pure if and only if $\rho^\natural$ splits.
\end{enumerate}
\end{corollary}

From Lemma~\ref{direct-limit-quotient-lemma} and Propositions~\ref{direct-limit-sf-proposition} and \ref{direct-limit-p-proposition} we can immediately deduce the important result
\begin{theorem}
Let $S$ be a monoid and let $\X$ be a class of $S-$acts closed under directed colimits. Then $\X$ is closed under chains of $\X-$pure congruences. In particular this is true for the classes $\X=\SF$ and $\X=\CP$.
\end{theorem}

Recall that an act $X$ is said to be {\em locally cyclic} if for all $x,y \in X$ there exists $z\in X, s,t\in S$ with $x=zs, y = zt$. By~\cite[Theorem 3.7]{renshaw-00} the indecomposable acts in $\CP$ and $\SF$ are the locally cyclic acts.

\begin{lemma}\label{condition-p-system-lemma}
Let $S$ be a monoid and suppose that $X$ satisfies condition $(P)$ and suppose we have a system of equations
$$
\begin{array}{rcl}
x_1s_1&=&x_2t_2\\
x_2s_2&=&x_3t_3\\
&\ldots&\\
x_{n-1}s_{n-1}&=&x_nt_n\\
\end{array}
$$
where $x_i\in X, s_i,t_i\in S$. Then there exists $y\in X, u_i\in S$ such that for $1\le i\le n-1$ we have $x_i=yu_i$ and $u_is_i=u_{i+1}t_{i+1}$.
\end{lemma}
\begin{proof}
We prove this by induction on $n$. Suppose then that $n=2$. Then our system is
$$
x_1s_1=x_2t_2
$$
and condition $(P)$ means there exists $y\in X, u_1,u_2\in S$ with $x_1=yu_1, x_2=yu_2$ and $u_1s_1=u_2t_2$ as required.

Suppose then that the result is true for $i\le n$ and suppose that we have a system of equations
$$
\begin{array}{rcl}
x_1s_1&=&x_2t_2\\
x_2s_2&=&x_3t_3\\
&\ldots&\\
x_{n-1}s_{n-1}&=&x_nt_n\\
x_ns_n&=&x_{n+1}t_{n+1}.\\
\end{array}
$$
By induction there exists $y\in X, u_i\in S$ such that for $1\le i\le n$ we have $x_i=yu_i$ and for $1\le i\le n-1$, $u_is_i=u_{i+1}t_{i+1}$. In addition, condition $(P)$ means there exists $y'\in X, u'_n,v'_n\in S$ with $x_n=y'u'_n, x_{n+1}=y'v'_n$ and $u'_ns_n=v'_nt_{n+1}$. But then $x_n=yu_n=y'u'_n$ and so there exists $z\in X, p,q\in S$ with $y=zp, y'=zq$ and $pu_n=qu'_n$. Hence for $1\le i\le n$ it follows that $x_i=z(pu_i)$ and for $1\le i\le n-1$,  $(pu_i)s_i=(pu_{i+1})t_{i+1}$. While $x_{n+1}=z(qv'_n)$ and $(pu_n)s_n = qu'_ns_n = (qv'_n)t_{n+1}$ as required.
\end{proof}

\smallskip
The following was suggested to us by Philip Bridge~\cite{bridge}. For a version involving more general categories see \cite{bridge-thesis}.

\begin{proposition}[{Cf. \cite[Theorem 5.21]{bridge-thesis}}]\label{fgt-proposition}
Let S be a monoid and suppose that $S$ satisfies the following property
$$
\forall s\in S\ \exists k\in\n\text{ such that }\forall m\in S\ |\{p\in S|ps=m\}|\le k.
$$
Then every $S-$act has an $\SF-$cover and a $\CP-$cover.
\end{proposition}
\begin{proof}
We show that every indecomposable $S-$act that satisfies condition $(P)$ (and hence every strongly flat indecomposable $S-$act) has a bound on its cardinality. Let X be an indecomposable $S-$act  which satisfies condition $(P)$. Then it is locally cyclic and so for all $x, y\in X$ there exists $z\in X, s, t\in S$ such that $x = zs, y = zt$.
$$
\begin{tikzpicture}[description/.style={fill=white,inner sep=2pt}]
\matrix (m) [matrix of math nodes, row sep=3em,
column sep=1.5em, text height=1.5ex, text depth=0.25ex]
{x&&y\\
 &z&\\};
\path[->]
(m-2-2) edge node[auto, left] {$s$} (m-1-1)
(m-2-2) edge node[auto,right] {$t$} (m-1-2);
\end{tikzpicture}
$$
Now we fix $x\in X$ and consider how many possible $y\in X$ could satisfy these equations. Firstly we take a fixed $s\in S$ and consider how many possible $z\in X$ could satisfy $x = zs$. By the hypothesis, there exists $k\in\n$ such that for any $m \in S$ $|\{p\in S : ps = m\}| \le k$. Let us suppose that there are at least $k + 1$ distinct $z$ such that $x = zs$. That is, $x = z_1s = z_2s = \ldots = z_{k+1}s$. Then by Lemma~\ref{condition-p-system-lemma} there exists $w \in X, p_1, \ldots ,p_{k+1}\in S$ such that $p_1s = \ldots = p_{k+1}s$ and $z_i = wp_i$ for each $i\in \{1, \ldots,k + 1\}$.
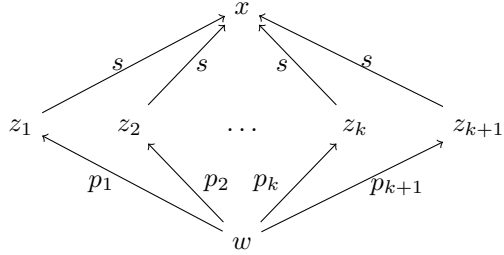
$$
\begin{tikzpicture}[description/.style={fill=white,inner sep=2pt}]
\matrix (m) [matrix of math nodes, row sep=3em,
column sep=2.5em, text height=1.5ex, text depth=0.25ex]
{&&x&&\\
 z_1&z_2&\ldots&z_k&z_{k+1}\\
 &&w&&\\};
\path[->]
(m-2-1) edge node[auto, left] {$s$} (m-1-3)
(m-2-2) edge node[auto,right] {$s$} (m-1-3)
(m-2-4) edge node[auto,left] {$s$} (m-1-3)
(m-2-5) edge node[auto,right] {$s$} (m-1-3)
(m-3-3) edge node[auto,left] {$p_1\ $} (m-2-1)
(m-3-3) edge node[auto,right] {$\ p_2$} (m-2-2)
(m-3-3) edge node[auto,left] {$p_k\ $} (m-2-4)
(m-3-3) edge node[auto,right] {$\ p_{k+1}$} (m-2-5);
\end{tikzpicture}
$$
However, by the hypothesis this means at least two $p_i$ are equal and hence at least two $z_i$ are equal which is a contradiction. So given some fixed $s\in S$ there are at most $k$ possible $z$ such that $x = zs$. Hence, there are no more than $\aleph_0|S|$ possible $z\in X, s\in S$ such that $x = zs$. Similarly, given a fixed $z\in X$, there are at most $|S|$ possible $t\in S$ such that $zt = y$ and hence there are no more than $\aleph_0|S|^2$ possible elements in $X$. So the result follows by Corollary \ref{lambda-skeleton-corollary}.
\end{proof}

A finitely generated monoid that satisfies this property is said to have {\em finite geometric type} (\cite{silva-04}). Let $B$ be the bicylic monoid and let $(s,t) \in B$. Suppose that $(m,n)\in B$ is fixed and suppose that $(p,q) \in B$ is such that $(p,q)(s,t)=(m,n)$. We count the number of solutions to this equation. Recall that
$$
(p,q)(s,t)=(p-q+\max(q,s),t-s+\max(q,s))=(m,n).
$$
If $q \ge s$ then $(p,q)=(m,n-(t-s))$ and there is at most one solution to the equation. Otherwise $(p,q)=(m-s+q,q)$ where $q$ ranges between $0$ and $s-1$. There are therefore at most $s+1$ possible values of $(p,q)$ that satisfy the equation and so $B$ has finite geometric type. Hence we deduce

\begin{proposition}
Let $S$ be the bicyclic monoid. Then all $S-$acts have an $\SF-$cover and a $\CP-$cover.
\end{proposition}

On letting $k=1$ in Proposition~\ref{fgt-proposition} we can deduce the following corollary.
\begin{corollary}
Let $S$ be a right cancellative monoid. Then every right $S-$act has an $\SF-$cover and a $\CP-$cover.
\end{corollary}

\smallskip

It also now follows that not every $\SF-$cover is a strongly flat cover as it was shown in~\cite[Remark 3.6]{renshaw-08} that $(\n,\cdot)$ is a monoid in which the 1-element act $\Theta$ does not have a strongly flat cover. It is however obviously right cancellative.

\medskip

Recall~\cite{isbell-71} that a monoid $S$ is said to satisfy condition $(A)$ if all right S-acts satisfy the ascending chain condition for cyclic subacts. This is equivalent to saying that every locally cyclic right $S-$act is cyclic.

\begin{proposition}
Let $S$ be a monoid that satisfies condition $(A)$. Then every right $S-$act has an $\SF-$cover and a $\CP-$cover.
\end{proposition}

\begin{proof}
By~\cite[Theorem 3.7]{renshaw-00} the indecomposable acts in $\CP$ and $\SF$ are the locally cyclic acts but since $S$ satisfies condition $(A)$ all the locally cyclic acts are cyclic. If $S/\rho$ is cyclic then clearly $|S/\rho|\le|S|$ and the result follows from Corollary~\ref{lambda-skeleton-corollary}.
\end{proof}

\smallskip

It is well known that not every monoid that satisfies condition $(A)$ is perfect and so we can then deduce that $\Pr-$covers are in general different from $\SF-$covers and $\CP-$covers.
Also given that indecomposable projective acts are cyclic then the indecomposable $S-$acts are bounded in size and so by Corollary~\ref{lambda-skeleton-corollary} we can deduce

\begin{proposition}
Let $S$ be a monoid. Every $S-$act has a $\Pr-$precover.
\end{proposition}

\begin{lemma}
Let $S$ be a monoid. If $A$ is a right $S-$act and if $k:C\to A$ is a $SF-$cover with $C$ projective, then $C$ is a $\Pr-$cover.
\end{lemma}
\begin{proof}
If $P$ is projective and if $g:P\to A$ is an $S-$map then $P$ is strongly flat and so there exists $h:P\to C$ with $kh=g$ and so $P$ is a projective cover.
\end{proof}

Since right perfect monoids satisfy condition $(A)$ then we have
\begin{corollary}
Let $S$ be a right perfect monoid. Then every right $S-$act has an $\SF-$cover.
\end{corollary}

In addition since $S$ is right perfect if and only if all strongly flat $S-$acts are projective we have

\begin{corollary}
$S$ is right perfect if and only if every right $S-$act has a projective $\SF-$cover.
\end{corollary}

\medskip

From~\cite[Example 2.9, Example 2.10]{khosravi-10} we can deduce
\begin{theorem}
The following classes of monoids satisfy condition $(A)$ and so every right $S-$act over such a monoid has an $\SF-$cover and a $\CP-$cover.
\begin{enumerate}
\item finite monoids,
\item rectangular bands with a 1 adjoined,
\item right groups with a 1 adjoined,
\item right simple semigroups with a 1 adjoined,
\item $(\n,\max)$.
\end{enumerate}
\end{theorem}

The previous results rely on us showing that the indecomposable strongly flat $S-$acts are bounded in size and hence the class of indecomposable strongly flat $S-$acts forms a set. We show there exists a monoid $S$ with a proper class of indecomposable strongly flat acts by constructing an indecomposable strongly flat act of arbitrary cardinality.

\begin{example}\label{fulltransform-example}
\rm Let $S=\mathcal{T}(\mathbb{N})$ be the full transformation monoid over the set of natural numbers and let $\phi:\n\times\n\to\n$ be a bijection of sets. For convenience, we write maps on the right. Given any set $X \neq \emptyset$, let $A_X=\{f:X \rightarrow \n\}$ be the set of all maps from $X$ to $\n$. We can make $A_X$ into an $S$-act by composition of maps - for $f\in A_X, s\in S$ define $fs\in A_X$ by $x(fs) = (xf)s$. Given any $f,g \in A_X$, let $h \in A_X$ be defined as $xh=(xf,xg)\phi$. Then define $u,v \in S$ to be $u=\phi^{-1}\pi_1$ and $v=\phi^{-1}\pi_2$, where $(x,y)\pi_1=x$ and $(x,y)\pi_2=y$. Therefore $f=hu$, $g=hv$ and $A_X$ is locally cyclic (hence indecomposable) and has cardinality at least $|X|$.
We now show $A_X$ is strongly flat. Let $f,g \in A_X$, $s,t \in S$ such that $fs=gt$. Define $h \in A_X$ as before, pick some $x \in X$ and define $u_x,v_x \in S$ by
\begin{align*}
nu_x=\left\{
\begin{array}{lr}
nu & \text{if $n \in $ im($h$)}\\
xf & \text{otherwise}
\end{array}\right. \\ \\
nv_x=\left\{
\begin{array}{lr}
nv & \text{if $n \in $ im($h$)}\\
xg & \text{otherwise.}
\end{array}\right.
\end{align*}
Then $f=hu_x$, $g=hv_x$ and $u_xs=v_xt$, so $A_X$ satisfies condition $(P)$.
Let $f \in A_X$, $s,t \in S$ such that $fs=ft$. Pick some $x \in X$ and define $w \in S$,
\begin{align*}
nw=\left\{
\begin{array}{lr}
n & \text{if $n \in $ im($f$)}\\
xf & \text{otherwise.}
\end{array}\right.
\end{align*}
Then $f=fw$ and $ws=wt$, so $A_X$ satisfies condition $(E)$ and is strongly flat.
\end{example}

\bigskip

Let $T$ be a monoid and let $S$ be a submonoid of $T$. If $X$ is an $S-$act that satisfies condition $(P)$ then $X\tensor_ST$ is a $T-$act and $X\to X\tensor_ST$ given by $x\mapsto x\tensor1$ is an $S-$monomorphism (since $X$ is flat). Moreover if $X$ is locally cyclic then so is $X\tensor_ST$ since if $x_1\tensor t_1,x_2\tensor t_2\in X\tensor_ST$ then there exists $z\in X, u_1,u_2\in S$ with $x_1 = zu_1, x_2=zu_2$ and so $x_1\tensor t_1=z\tensor u_1t_1 = (z\tensor1)u_1t_1$ and similarly $x_2\tensor t_2=(z\tensor1)u_2t_2$.

Finally we can also deduce that $X\tensor_ST$ satisfies condition $(P)$ as if $(x\tensor t_1)r_1=(x'\tensor t_2)r_2$ then there exist $x_2,\ldots, x_n\in X, u_2,\ldots, u_n,v_2\ldots, v_n\in S, p_2,\ldots p_{n-1}\in T$ such that
$$
\begin{array}{rclcrcl}
x&=&x_2u_2&\quad&u_2t_1r_1&=&v_2p_2\\
x_2v_2&=&x_3u_3&\quad&u_3p_2&=&v_3p_3\\
&&&\ldots&\\
x_{n-1}v_{n-1}&=&x_nu_n&\quad&u_np_{n-1}&=&v_nt_2r_2\\
x_nv_n&=&x'\\
\end{array}
$$
and so by Lemma~\ref{condition-p-system-lemma} there exists $y\in X$ and $w_i\in S$ such that $x_i=yw_i$ and $w_iv_i=w_{i+1}u_{i+1}$ ($x=yw_1, w_1=w_2u_2$ and $x'=yw_{n+1}, w_nv_n=w_{n+1}$) and so we have a scheme of the form
$$
\begin{array}{rclcrcl}
x=yw_1&=&yw_2u_2&\quad&u_2t_1r_1&=&v_2p_2\\
yw_2v_2&=&yw_3u_3&\quad&u_3p_2&=&v_3p_3\\
&&&\ldots&\\
yw_{n-1}v_{n-1}&=&yw_nu_n&\quad&u_np_{n-1}&=&v_nt_2r_2\\
yw_nv_n&=&yw_{n+1}=x'\\
\end{array}
$$
Hence $x\tensor t_1=(y\tensor1)w_1t_1, x'\tensor t_2=(y\tensor1)w_{n+1}t_2$ and 
$$
(w_1t_1)r_1 = w_2u_2t_1r_1 = w_2v_2p_2=w_3u_3p_2=w_3v_3p_3=\ldots=w_{n-1}v_{n-1}p_{n-1}=
$$
$$
w_nu_np_{n-1}=w_nv_nt_2r_2=(w_{n+1}t_2)r_2.
$$

In a similar way, if $X$ is strongly flat then whenever $(x\tensor t)r_1=(x\tensor t)r_2$ in $X\tensor_ST$ we can proceed as above and deduce the existence of a scheme
$$
\begin{array}{rclcrcl}
x=yw_1&=&yw_2u_2&\quad&u_2tr_1&=&v_2p_2\\
yw_2v_2&=&yw_3u_3&\quad&u_3p_2&=&v_3p_3\\
&&&\ldots&\\
yw_{n-1}v_{n-1}&=&x_nu_n&\quad&u_np_{n-1}&=&v_ntr_2\\
yw_nv_n&=&yw_{n+1}=x.\\
\end{array}
$$
Now since $yw_1=yw_{n+1}$ and since $X$ satisfies condition $(E)$ then there exists $z\in X, u\in S$ with $y=zu$ and $uw_1=uw_{n+1}$ and so $x\tensor t = (z\tensor1)uw_1t$ and as before $(uw_1t)r_1=\ldots=uw_{n+1}tr_2=(uw_1t)r_2$ and so $X$ satisfies condition $(E)$ as well.

\smallskip

Let $T$ be a monoid that satisfies condition $(A)$, let $S$ be a left pure submonoid of $T$ (in the sense that the inclusion $S\to T$ is a left pure $S-$monomorphism) and let $X$ be a locally cyclic right $S-$act. Then from above we see that $X\tensor_ST$ is a locally cyclic right $T-$act and so is cyclic. Hence there exists $x_0\in X, t_0\in T$ such that $X\tensor_ST\cong(x_0\tensor t_0)T$. We show that $X$ is also cyclic. First, we say that a left $S-$monomorphism $f:C\to D$ is {\em stable} if for all right $S-$monomorphisms $\lambda:A\to B$
$$
\im(1_B\tensor f)\cap\im(\lambda\tensor1_D) = \im(\lambda\tensor f).
$$
It was shown in~\cite[Theorem 3.1]{renshaw-91} that left pure monomorphisms are stable. In particular the above remarks hold when $\lambda:x_0S\to X, f:S\to T$ are the inclusions. Consequently if $x\in X$ then $x\tensor 1=x_0\tensor t$ in $X\tensor_S T$ for some $t\in T$. Hence there exists $s\in S$ such that $x\tensor1=x_0s\tensor1$ in $X\tensor_ST$ and since $X\to X\tensor_ST$ is a monomorphism, by left purity of $S\to T$, then $x=x_0s$ as required.

\smallskip

Hence we can deduce
\begin{proposition}
The class of monoids that satisfy condition $(A)$ is closed under the taking of left pure submonoids.
\end{proposition}

\smallskip
We can also deduce

\begin{theorem}
Let $T$ be a monoid and let $\X_T=\SF_T$ or $\X_T=\CP_T$. Let ${\cal M}$ be the class of monoids such that for all $T\in{\cal M}$ there exists a cardinal $\kappa$ with $|X|<\kappa$ for all locally cyclic right $T-$acts $X\in\X_T$. Then ${\cal M}$ is closed under submonoids. In addition for any monoid $S\in{\cal M}$ every right $S-$act has an $\SF-$cover and a $CP-$cover.
\end{theorem}

\begin{proof}
Let $T\in{\cal M}$ and let $S$ be a submonoid of $T$. If $X\in\X_S$ is a locally cyclic right $S-$act then $X\tensor_ST\in\X_T$ is a locally cyclic right $T-$act. By assumption there exists a cardinal $\kappa$ such that $|X\tensor_ST|<\kappa$ and so since $X\to X\tensor_ST$ is a monomorphisms then $|X|<\kappa$ and hence $S\in{\cal M}$.
\end{proof}

\begin{corollary}
Let $S$ be any submonoid of the bicyclic monoid. Then every $S-$act has an $\SF-$cover and a $\CP-$cover.
\end{corollary}

\medskip

Many of the results in this paper involve monoids belonging to ${\cal M}$. However Example~\ref{fulltransform-example} demonstrates that ${\cal M}$ is not the class of all monoids. One of the proofs of the flat cover conjecture in \cite{bican-01} involved showing that every module over a unitary ring satisfied a condition very similar to that given in Theorem~\ref{congruence-pure-theorem}. We feel that a similar situation should hold in the category of $S-$acts.

\smallskip

We hope to consider the classes of torsion free, divisible, injective and free acts in a subsequent paper.

\end{document}